\documentclass[12pt]{article}
\usepackage{amsmath,amssymb}
\usepackage{amsthm}

\def\MC{\mathcal{M}}

\def\C{\mathbf{C}}
\def\E{\mathbf{E}}
\def\L{\mathbf{L}}

\def\R{\mathbf{R}}

\def\1{\mathbf{1}}

\def\tr{\rm{tr}}

\def\al{\alpha}

\def\pa{\partial}
\def\ep{\epsilon}
\def\de{\delta}
\def\ga{\gamma}
\def\ka{\varkappa}

\newtheorem{theorem}{Theorem}[section]

\newtheorem{remark}{Remark}

\newcommand{\Ga}{\Gamma}
\newcommand{\De}{\Delta}

\begin{document}
\title{Quantum mean-field games}

\author{Vassili N. Kolokoltsov\thanks{Department of Statistics, University of Warwick,
 Coventry CV4 7AL UK, associate member of HSE, Moscow,
  Email: v.kolokoltsov@warwick.ac.uk}}
\maketitle
	
\begin{abstract}
Quantum games represent the really 21st century branch of game theory, tightly linked to the modern
development of quantum computing and quantum technologies. The main accent in these developments so far was made on
stationary or repeated games. In the previous paper of the author the truly dynamic quantum game theory was initiated
 with strategies chosen by players in real time. Since direct continuous observations are known to destroy quantum evolutions
 (so-called quantum Zeno paradox) the necessary new ingredient for quantum dynamic games represented the theory of
 non-direct observations and the corresponding quantum filtering. Another remarkable 21st century branch of game theory
  represent the so-called mean-field games (MFG), with impressive and ever growing development.
 In this paper we are merging these two exciting new branches of game theory.
 Building a quantum analog of MFGs requires the full reconstruction of its foundations and methodology, because in $N$-particle quantum
 evolution particles are not separated in individual dynamics and the key concept of the classical MFG theory, the empirical measure defined as the sum of Dirac masses of the positions of the players, is not applicable in quantum setting.
 As a preliminary result we derive the new nonlinear stochastic Schr\"odinger equation, as the limit of continuously observed
 and controlled system of large number of interacting quantum  particles, the result that may have
 an independent value. We then show that to a control quantum system of interacting particles there corresponds a special
 system of classical interacting particles with the identical limiting MFG system, defined on an appropriate Riemanian
  manifold. Solutions of this system are shown to specify approximate Nash equilibria for $N$-agent quantum games.
\end{abstract}



\section{Introduction}

Quantum games represent the really 21st century branch of game theory, tightly linked to the modern
development of quantum computing and quantum technologies. Initiated by Meyer \cite{MeyerD99},
Eisert,  Wilkens and Lewenstein \cite{EWL99}, and Marinatto and Weber \cite{MW00}, the theory now
boasts of many beautiful results obtained by various authors in numerous publications, see e.g.
surveys \cite{KhanReview18}, \cite{GuoZhang08}, and a mathematically oriented survey \cite{KolQSurv}.
 However, the main accent in these developments was made on stationary or repeated games.
 In \cite{KolDynQuGames} the author developed the truly dynamic quantum game theory with strategies
 chosen by players in real time. Since direct continuous observations are known to destroy quantum
 evolutions (so-called quantum Zeno paradox) the necessary new ingredient for quantum dynamic games
 represented the theory of non-direct observations and the corresponding quantum filtering. This
 theory was essentially developed by Belavkin in the 80s of the last century, in \cite{Bel87},
 \cite{Bel88}, \cite{Bel92}, see \cite{BoutHanJamQuantFilt} for a readable modern account. There is
 an important work under way on the technical side of organising feedback quantum control in real time,
 see e.g. \cite{Armen02Adaptive}, \cite{Bushev06Adaptive} and \cite{WiMilburnBook}.

 Another recently emerged branch of game theory represent the so-called mean-field games (MFG).
They were initially introduced by M. Huang, R. Malham\'e, P. Caines in \cite{HCM3} and by
J-M. Lasry, P-L. Lions in \cite{LL2006} and have an impressive and ever growing development,
see e.g. recent monographs \cite{BenFr},  \cite{CarDelbook18}, \cite{Gomesbook},
 \cite{KolMalbook19} and references therein.

 In this paper we are merging these two exciting new branches of game theory.
 Building a quantum analog of MFGs requires the full reconstruction of its foundations and methodology,
 because in $N$-particle quantum evolution particles are not separated in individual dynamics and the
 key concept of the classical MFG theory, the empirical measure defined as the sum of Dirac masses of
 the positions of the players is not applicable in quantum setting.
 As a preliminary result we derive (first heuristically and then rigorously)
 the new nonlinear stochastic Schr\"odinger equation, as the limit
 of continuously observed and controlled system of a large number of interacting quantum  particles,
 the result that may have an independent value. There is a huge literature on the derivation of
 deterministic nonlinear equations like Hartree, Gross-Pitaevskii equations, see some review in
 \cite{GolsePaul}. There is a parallel development on the mathematical properties of various nonlinear
 Schr\"odinger equations including stochastic and controlled ones,  see e.g. \cite{BarbRock16},
 \cite{BarbRock18}, \cite{Brzes14}, \cite{GreckGenerNonlinSchr} and references therein. Our equation
 is different from the equations considered in these papers, as the linearity depends on the expectations
 of the correlations calculated with respect to the solution. It arises naturally from continuously
 observed systems and is reminiscent to the McKean-Vlasov nonlinear diffusion. It can be in fact considered
 as some quantum analog of the latter.

 Motivated by this result, we build a correspondence between quantum $N$-agent dynamic games and classical
 $N$-player dynamic games on appropriate Riemannian manifolds in such a way that the corresponding games have
 the same limiting MFG system that describes these games in the limit of infinite number of agents.
 Of course the precise link between the limiting game and the pre-limit $N$-agent games is quite different
 for quantum and classical games. Our main result shows that, similar to the classical setting, solutions
  to the limiting MFG system specify approximate $\ep$-Nash equilibria for $N$-agent quantum games,
  with $\ep$ of order $N^{-1/4}$, which is of quite  different from classically available rates of convergence
  of type $\ep\sim N^{-1/2}$ (see  \cite{HCM3}) or $\ep \sim N^{-1}$ (see \cite{KolTrYa14}).

The content of the paper is as follows. In the next section we recall the basic theory of quantum
continuous measurement and filtering. In section \ref{secmynonlinSchrod} our new nonlinear equations are
introduced in a heuristic manner and the rigorous convergence results are formulated.
The result concerning controlled dynamics is seemingly new even in the deterministic case.

In Section \ref{secmainres} the limiting forward-backward MFG system is introduced
and the main result of the paper is formulating stating
that the solutions to the limiting forward-backward systems (if they exist) determine
the $\ep$-Nash equilibria for the corresponding $N$-agent quantum game. The next three
 sections are devoted to the proof of the three main theorems.

Section \ref{seclimforback} is mostly independent from the rest of the paper and
 can be looked at as an introduction to classical MFGs on compact Riemannian manifolds on the example
 of complex projective spaces, which play the role of the state spaces for finite-dimensional
 quantum mechanics. Based on the discovery from \cite{KolDynQuGames},
  that allows one to organise special homodyne detection schemes of continuous observation
 on finite-dimensional quantum systems in such a way that the resulting diffusion operator turns to the
  standard Laplace-Beltrami operator on a complex projective space, we prove under these arrangements
  the global existence and local well-posedness of the limiting forward-backward MFG system on complex
  projective spaces thus supplying the missing existence result
 from our main theorem at least for finite-dimensional quantum games.

Final Section \ref{secconc} presents some questions arising from our analysis.

\section{Prerequisites: nondemolition observation and quantum filtering}
\label{secnondenmes}

The general theory of quantum non-demolition observation, filtering and resulting
feedback control was built essentially in papers  \cite{Bel87}, \cite{Bel88}, \cite{Bel92}.
A very good readable introduction is given in \cite{BoutHanJamQuantFilt}.
 We shall describe briefly the main result of this theory.

The non-demolition measurement of quantum systems can be organised in two versions:
photon counting and homodyne detection. One of the first mathematical results on the control
with photon counting measurement was given in \cite{Kol92}, which can be used to develop the
corresponding  game theoretical version. But here we fully concentrate on the homodyne
(mathematically speaking, diffusive type) detection. Under this type of measurement the
output process $Y_t$ is a usual Brownian motion (under appropriate probability distribution).
There are several (by now standard) ways of writing down the quantum filtering equation for
states resulting from the outcome of such process. The one which is the most convenient to our
purposes is the following linear Belavkin filtering equation (which is a particular version
of the stochastic Schr\"odinger equation)
describing the a posteriori (pure but not normalized) state:

\begin{equation}
\label{eqqufiBlin}
d\chi =-[iH\chi +\frac12 L^*L \chi ]\,dt+L\chi dY_t,
\end{equation}
where the unknown vector $\chi$ is from the Hilbert space of the observed quantum system,
which we shall sometimes referred to as the atom,
the self-adjoint operator $H$ is the Hamiltonian
of the corresponding initial (non-observed) quantum evolution and the operator $L$
is the coupling operator of the atom to the optical measurement device specifying the chosen version
of the homodyne detection.


The initial derivation of the quantum filtering equation was carried out via the method
of quantum stochastic calculus. Later on more elementary derivations appeared. It can be obtained
from an appropriate limit of sequential discrete observation scheme, see e.g. \cite{BelKol} or
\cite{Pellegrini}. A derivation from the theory of instruments was given in \cite{BarchBel}
 and \cite{Holevo91}.

An important role in the theory is played by the so-called innovation process
\begin{equation}
\label{eqdefinnov}
dB_t=dY_t-\langle L+L^* \rangle_{\chi} \, dt,
\end{equation}
where for an operator $A$ and a vector $v$ in a Hilbert space we use the (more or less standard)
notation for the average value of $A$ in $v$:
\begin{equation}
\label{eqdefmeanvalueoper}
\langle A \rangle_v=(v,Av)/(v,v).
\end{equation}

The innovation process is in some sense a more natural driving noise to deal with,
because it turns out to be the standard Brownian motion (or the Wiener process)
with respect to the fixed (initial vacuum) state of the homodyne detector, while
the output process $Y_t$ is a Brownian motion with respect to the states transformed
by the (quite complicated) interaction of the quantum system and optical device,
which can also be obtained by the Girsanov transformation from the innovation process $B_t$.
Due to \eqref{eqdefinnov}, $dY_t$ satisfies the usual Ito rule: $dY_tdT_t=dt$, which is
the basic tool in all calculations.

A very particular case represent the equations with anti-Hermitian operators $L$: $L^*=-L$.
As seen from \eqref{eqdefinnov}, in this case the innovation process coincides with the output process,
which thus becomes the standard Brownian motion on its own. This means that the noise does not
properly interact with the atom, and therefore this case is the less interesting
for continuous measurement, see discussion in \cite{BarchBook}. Nevertheless the filtering theory
still applies to this case and control can be analysed via the averaging with respect to the noise.
This case is referred to in the theory as conservative, because in this case (as seen from direct application
of Ito's formula) solutions to \eqref{eqqufiBlin} preserve the norm almost surely, that is, the resolving
operators for the Cauchy problem of this equation are unitary almost surely. In the present paper the rigorous
analysis will be carried out only for the conservative case.

The theory extends naturally to the case of several, say $N$, coupling operators $\{L_j\}$,
where the quantum filtering is described by the following direct extension
of equation \eqref{eqqufiBlin}:
 \begin{equation}
\label{eqqufiBlinn}
d\chi =-[iH\chi +\frac12 \sum_j L_j^*L_j \chi ]\,dt+\sum_j L_j\chi dY^j_t,
\end{equation}
with the $N$-dimensional output process $Y_t=\{Y^j_t\}$.
The corresponding innovation process is the standard $N$-dimensional
Wiener process with the coordinate differentials
\[
dW^j_t=dY^j_t-\langle L_j+L_j^* \rangle_{\chi} \, dt.
\]

Recall that the density matrix or density operator $\ga$ corresponding to a unit vector
$\chi\in L^2(X)$ is defined as the orthogonal projection operator on $\chi$.
This operator is usually expressed either as the tensor product  $\ga=\chi\otimes \bar \chi$
or in the most common for physics bra-ket Dirac's notation as $\ga=|\chi\rangle \langle \chi|$.
Of course in the tensor notation $\ga$ is formally an element of the tensor product $L^2(X^2)$.
However, considered as an integral kernel, it is identified with the corresponding integral operator.

As one checks by direct application of Ito's formula,
in terms of the density matrix $\ga$, equation
\eqref{eqqufiBlin} rewrites as
\begin{equation}
\label{eqqufiBlindens}
d\ga=-i[H,\ga]\, dt
+(L\ga L^* -\frac12 L^*L \ga -\frac12 \ga L^*L)\, dt
+(\ga L^*+L \ga) dY.
\end{equation}
In particular, the expectation $\E \ga$ satisfies the following master equation
(sometimes referred to as the Lindblad equation)
\begin{equation}
\label{eqqufiBlindensexp}
d\E \ga=-i[H,\E \ga]
+(L\E \ga L^* -\frac12 L^*L \E \ga -\frac12 \E \ga L^*L)\, dt.
\end{equation}

The theory of quantum filtering reduces the analysis of quantum dynamic control and games
to the controlled version of evolutions \eqref{eqqufiBlinn}. The simplest situation concerns the case
 when the homodyne device is fixed, that is the operators $L_j$ are fixed, and the players can control the
 Hamiltonian $H$, say, by applying appropriate electric or magnetic fields to the atom. Thus equation
 \eqref{eqqufiBlinn} becomes modified by allowing $H$ to depend on one or several control parameters.
 One can even prove a rigorous mathematical result, the so-called separation principle
 (see \cite{BoutHanQuantumSepar}), that shows that the effective control of an observed quantum system
 (that can be based in principle on the whole history of the interaction of the atom and optical devices)
 can be reduced to the Markovian feedback control of the quantum filtering equation, with the feedback
 at each moment depending only on the current (filtered) state of the atom.

\section{A new nonlinear stochastic Schr\"odinger equation}
\label{secmynonlinSchrod}

The well developed theory of the so-called nonlinear stochastic Schr\"odinger
 equations deals with the stochastic equations of the type
\[
d\psi(t)=-iH \psi dt -i f(t, \psi(t)) dt -g(t,\psi(t)) dW(t)
\]
with a Hamiltonian operator $H$, some nonlinear function $f,g$ and various noises $dW$ (including multidimensional),
see e.g. \cite{BarbRock16}, \cite{BarbRock18}, \cite{Brzes14}, \cite{GreckGenerNonlinSchr} and references therein.

For our theory a different type of  nonlinear equation is needed, with nonlinearity depending
additionally on the distribution specified by the wave function solving the equation. It bears
some analogy to the classical McKean-Vlasov diffusions,
though the role of the law of the diffusion is now played by its quantum analog.

To see where it comes from, let $X$ be a Borel space with a fixed Borel measure that we denote by $dx$, and let
us consider the quantum evolution of $N$ particles
driven by the standard interaction Hamiltonian
\begin{equation}
\label{eqHambinaryinter}
 H(N)f(x_1, \cdots , x_N)=\sum_{j=1}^N H_jf(x_1, \cdots , x_N)
 + \frac{1}{N}\sum_{i<j\le N} V(x_i, x_j)f(x_1, \cdots , x_N),
\end{equation}
 where $x_j\in X$, $H_j$ is the Hamiltonian $H$ of a single particle
 (that is, $H$ is a self-adjoint operator in $L_2(X)$)
  applied to the $j$th variable of a function $f(x_1, \cdots ,x_N)$
 and the interaction potential $V(x,y)$  is a symmetric function of two variables,
 and observed via symmetric coupling with one-dimensional (one-particle) operators.
 That is, we consider the filtering equation \eqref{eqqufiBlinn} of the type
\begin{equation}
\label{eqmainNpartBel}
d\Psi_{N,t} =-[iH(N)\Psi +\frac12 \sum_{j=1}^N L_j^*L_j \Psi_{N,t} ]\,dt
+\sum_{j=1}^N L_j\Psi_{N,t} \, dY^j_t,
\end{equation}
where $L_j$ denote the identical one-particle operators $L$ ($L$ is an operator
 in $L^2(X)$) acting on the $j$th variable of the wave function $\Psi_{N,t}$.
The density matrix $\Ga_{N,t}=\Psi_{N,t}\otimes \overline{\Psi_{N,t}}$ satisfies the
corresponding equation of type \eqref{eqqufiBlindens}:
\[
d \Ga_{N,t}
=-i\sum_j [H_j,\Ga_N]-\frac{i}{N} \sum_{l<j\le N} [V_{lj},\Ga_N]
\]
\begin{equation}
\label{eqmainNpartBeldens}
+\sum_j (L_j\Ga_N L_j^* -\frac12 L^*_jL_j \Ga_N -\frac12 \Ga_N L^*_jL_j)\, dt
+\sum_j (\Ga_N L^*_j+L_j \Ga_N) dY_j,
\end{equation}
where $V_{lj}$ denote the operator of multiplication by $V(x_l,x_j)$.

For the application three cases of the space $X$ are of major interest: $X$
a finite set with the standard uniform measure (all points have measure $1$),
$X=\R^d$ with Lebesgue measure and $X$ a compact Riemannian manifold with
the Riemann-Lebesgue volume as the measure.

Recall the standard heuristic argument introducing the nonlinear Schr\"odinger or Hartree equation
in the deterministic case, that is with $L=0$. Assume that the solution $\Psi_{N,t}$
can be written approximately as the product of individual functions $\psi (x_j)$ with the same $\psi$ for all $j$.
(This is the weakest point of the heuristics, as it is not at all obvious, in which sense such an approximation may hold.)
 Each $j$th particle is influenced by the average interaction potential
 $V(x_j,x_m)$ over the position of all particles $m\neq j$. But the distribution of the position of an
 $m$th particle is given by the density $|\psi (x_m)|^2$. Thus the total potential acting on $j$th particle is
 \[
 \frac{1}{N}\sum_{m\neq j}\int V(x_j,x_m) |\psi(x_m)|^2 dx_m=\frac{N-1}{N}\int V(x_j,y) |\psi(y)|^2 dy.
 \]
 For large $N$ this equals approximately to $\int V(x_j,y) |\psi(y)|^2 dy$.
 With such interaction potential the Schr\"odinger equation for each particle gets the form
  \begin{equation}
  \label{eqHartree}
 \dot \psi_t(x) =-iH \psi_t(x) -i (V^{|\psi_t|^2})(x) \psi_t (x),
 \end{equation}
 where $V^{|\psi|^2}$ denotes the function $\int V(x,y) |\psi(y)|^2 dy$,
 which is the standard nonlinear Schr\"odinger or the Hartree equation.

 Very often $V$ is assumed to depend on the difference of the arguments,
 that is, to be of the form $V(x-y)$ with an even function $V$.
 In this case the Hartree equation takes its most familiar form
 \[
 \dot \psi_t(x) =-iH \psi_t(x) -i (V\star |\psi_t|^2)(x) \psi_t (x),
 \]
 where
 \[
 (V\star |\psi|^2)(x)=\int V(x-y) |\psi(y)|^2 dy
 \]
 denotes the convolution.

 In the case of evolution \eqref{eqqufiBlinn} there can be no question of having identical wave functions
 in the product, because they are controlled by different noises.
 However, assuming the initial condition is the product of identical functions,
 and that noises $Y_j$ are independent, we may assume that the terms in the product $\psi_j(x_j)$ are
 independent and identically distributed and hence, by the strong law of large number, there may exist an
 almost sure deterministic limit
 \[
 \xi_t(x)=\lim \frac{1}{N} \sum_{j=1}^N |\psi_{j,t}(x)|^2 =\E |\psi_{j,t}(x)|^2.
 \]
Relating the observations given by $L_j$ to the evolution of $j$th particle,
 we may suggest the following limiting equations for the individual particles:
 \begin{equation}
\label{eqmainnonlinpartBel}
d\psi_{j,t}(x) =-[iH_j\psi_{j,t}(x) +i(V^{\xi_t})(x) \psi_{j,t}(x)
+\frac12  L^*L \psi_{j,t} ]\,dt+L_j\psi_{j,t} \, dY^j_t,
\quad \xi_t(x)=\E |\psi_{j,t}(x)|^2.
\end{equation}

Here, as in \eqref{eqHartree},
$V^{\xi_t}$ denotes the function $\int V(x,y) \xi_t(y) \, dy$.

The corresponding equation for the density matrix \eqref{eqqufiBlindens} has the form
\[
d\ga_{j,t}=-i[H,\ga_{j,t}] \, dt-i[V^{\xi_t}, \ga_{j,t}] \, dt
+(L\ga_{j,t} L^* -\frac12 L^*L \ga_{j,t} -\frac12 \ga_{j,t} L^*L)\, dt
\]
\begin{equation}
\label{eqmainnonlinpartBeldensV}
+(\ga_{j,t} L^*+L \ga_{j,t}) dY^j_t, \quad \xi_t(x)=\E \ga_{j,t}(x,x),
\end{equation}
where, with usual abuse of notation, $V^{\xi_t}$ is considered as the operator of multiplication
by the function $V^{\xi_t}$.

 Below we shall give a rigorous derivation of \eqref{eqmainnonlinpartBel} from
\eqref{eqmainNpartBel} under the strong simplifying assumption of conservativity,
namely under the condition $L^*=-L$, in which case the output processes $Y_j$ are
standard independent Brownian motions (see Section \ref{secnondenmes}).

    Apart from the pairwise interaction expressed by a multiplication operator, another standard class
    of binary interactions (specifically often used in finite-dimensional quantum mechanics)
    is expressed by integral operators $A$ with kernels $A(x,y;x',y')$ that act
    on the functions of two variables as
\[
A\psi(x,y)=\int_{X^2}   A(x,y;x',y') \psi(x',y') \, dx' dy'.
\]
It is usually assumed (and we shall do it) that $A$ are symmetric in the sense that they take
symmetric functions $\psi(x,y)$ (symmetric with respect to permutation of $x$ and $y$) to
symmetric functions. For the kernels this means that they are symmetric with respect to the
simultaneous exchange of the first pair of variables and the second one:
\[
A(x,y;x',y')=A(y,x;y',x').
\]

 With such interaction the $N$ particle Hamiltonian becomes
\begin{equation}
\label{eqHambinaryinter1}
 H(N)f(x_1, \cdots , x_N)=\sum_{j=1}^N H_jf(x_1, \cdots , x_N)
 + \frac{1}{N}\sum_{i<j\le N} A_{ij}f(x_1, \cdots , x_N),
\end{equation}
with $A_{ij}$ denoting the operator $A$ acting on the variables $x_i,x_j$ of $f$,
and the equation for the density matrix $\Ga_{N,t}=\Psi_{N,t}\otimes \overline{\Psi_{N,t}}$
is given by \eqref{eqmainNpartBeldens} with $A_{lj}$ instead $V_{lj}$.

The corresponding analog of nonlinear equation \eqref{eqmainnonlinpartBel} can be written in the form
\begin{equation}
\label{eqmainnonlinpartBel1}
d\psi_{j,t}(x) =-[iH_j\psi_{j,t}(x) +iA^{\bar \eta_t} \psi_{j,t}(x)
+\frac12  L^*L \psi_{j,t}(x) ]\,dt+L_j\psi_{j,t}(x) \, dY^j_t,
\end{equation}
where $A^{\bar \eta_t}$ is the integral operator in $L^2(X)$ with the integral kernel
\[
A^{\bar \eta_t}(x;y)=\int_{X^2} A(x,y;x',y') \overline{\eta_t (y,y')} \, dydy'
\]
and
\[
\eta_t (y,z)=\E (\psi_{j,t} (y) \bar \psi_{j,t}(z)),
\]
with the equation for the density matrix \eqref{eqqufiBlindens} being
\[
d\ga_{j,t}=-i[H,\ga_{j,t}] \, dt-i[A^{\bar \eta_t}, \ga_{j,t}]
+(L\ga_{j,t} L^* -\frac12 L^*L \ga_{j,t} -\frac12 \ga_{j,t} L^*L)\, dt
\]
\begin{equation}
\label{eqmainnonlinpartBel1dens}
+(\ga_{j,t} L^*+L \ga_{j,t}) dY^j_t, \quad \eta_t(y,z)=\E (\psi_{j,t} (y) \bar \psi_{j,t}(z))=\E \ga_{j,t}(y,z).
\end{equation}

For the expectation $\eta_t(y,z)= \E \ga_{j,t}(y,z)$ we get the following
nonlinear version of the Lindblad equations:
 \begin{equation}
\label{eqmainnonlinpartBel1densex}
d\eta_t=-i[H,\eta_t] \, dt-i[A^{\bar \eta_t}, \eta_t]
+(L\eta_t L^* -\frac12 L^*L \eta_t -\frac12 \eta_t L^*L)\, dt.
\end{equation}

The density matrix $\eta_t$ can be considered as the quantum analog of the empirical measure
of classical particles, and equation \eqref{eqmainnonlinpartBel1densex} as the quantum analog
of the McKean-Vlasov equation of nonlinear diffusion.

Formally the case of the multiplication operators by $V(x-y)$ can be considered as
a particular case of the integral operator with the singular kernel
\[
V(x-y) \de(x-x')\de(y-y').
\]
In finite dimensional setting this just means that the matrix of the corresponding integral operator is diagonal.

The story extends naturally to the case when the measurement related to each particle is multi-dimensional,
that is the operator $L$ is vector-valued, $\L=(L^1, \cdots, L^k)$, in which case each noise $dY^j_t$ is also
$k$-dimensional, so that the term $\L_j\psi_{j,t} \, dY^j_t$ should be understood as the inner product,
\[
\L_j\psi_{j,t} \, dY^j_t=\sum_{l=1}^k L_j^l\psi_{j,t} \, dY^j_{l,t},
\]
with all other terms containing $L$ understood in the same way.

For the rigorous derivation of the Hartree equation \eqref{eqHartree} from the corresponding
$N$-particle evolution several ingenues methods were developed recently, see a review in \cite{GolsePaul}.
Our analysis of the stochastic situation will be carried out via the method suggested by Pickl,
see \cite{Pickl} and  \cite{KnowlesPickl}, appropriately adapted and modified to address the stochastic setting.
In Pickl's approach the main measures of the deviation of the solutions $\Psi_{N,t}$ to $N$-particle
systems from the product of the solutions to the Hartree equations are the following positive numbers
from the interval $[0,1]$:
\[
E^{(k)}_{N,t} =1-(\psi_t^{\otimes k}, \Ga_{N,t} \psi_t^{\otimes k})
\]
and in particular,
\[
\al_N(t)=E^{(1)}_{N,t}.
\]
Clearly, if $\Ga_{N,t}$ were the tensor product of $\psi_t$, then one would have $E^{(k)}_{N,t}=0$.
Hence the convergence $E^{(k)}_{N,t}\to 0$, as $N\to \infty$, expresses some kind of convergence
of $\Ga_{N,t}$ to the product state.
As was shown in \cite{KnowlesPickl},
\begin{equation}
E^{(k)}_{N,t}\le k E^{(1)}_{N,t},
\end{equation}
so that for the convergence of all $E^{(k)}_{N,t}$ it is sufficient to show the convergence
$\al_N(t)\to 0$.

In the present stochastic case, the quantities $E^{(k)}_{N,t}$ depend not just on the number $k$ of particles
in the product, but on the concrete choice of these particles. For instance, the proper stochastic analog
of the quantity $\al_N(t)$ is the collection of random variables
\begin{equation}
\label{eqforalpha}
\al_{N,j}(t)=  1-(\psi_{j,t}, \Ga_{N,t} \psi_{j,t})=1-{\tr}(\ga_{j,t} \Ga_{N,t}).
\end{equation}
Here $\ga_{j,t}$ is identified with the operator in $L^2(X^N)$ acting on the $j$th variable.

Due to the i.i.d. property of the solutions to  \eqref{eqmainnonlinpartBel} or \eqref{eqmainnonlinpartBel1},
the expectations $\E E^{(k)}_{N,t}$ and in particular $\E \al_N(t)=\E \al_N(j,t)$ are well defined (they do not
depend on a particular choice of $k$ particles).

Expressions $\al_{N,j}$ can be linked with the traces by the following inequalities, due to Knowles and Pickl:
\begin{equation}
\label{ineqKnPi}
\al_{N,j}(t)\le {\tr} |\Ga^{(j)}_{N,t}-\ga_{j,t}|\le 2\sqrt{2\al_{N,j}(t)},
\end{equation}
where $\Ga^{(j)}_{N,t}$ is the partial trace of $\Ga_{N,t}$ with respect to all variables except for the $j$th,
see Lemma 2.3 from \cite{KnowlesPickl}.

The following result shows that heuristical arguments given above can be corroborated
by the rigorous analysis.

\begin{theorem}
\label{thmynonlinSch}
Let the operators $H,L$ be bounded and the interaction be given either by the multiplication by a bounded
symmetric function $V(x,y)\in L_{\infty}(X^2)$ or by a symmetric self-adjoint integral operator $A$ with a Hilbert-Schmidt
kernel, that is a kernel $A(x,y;x',y')$ such that
\begin{equation}
\label{eq1thmynonlinSch}
\|A\|^2_{HS}=\int_{X^4} |A(x,y;x',y')|^2 \, dx dy dx'dy' <\infty,
\end{equation}
\begin{equation}
\label{eq1athmynonlinSch}
A(x,y;x'y')=A(y,x;y',x'), \quad A(x,y;x',y')=\overline{A(x',y';x,y)}.
\end{equation}

Let $L$ be anti-Hermitian, $L^*=-L$, and the noises $Y_j$ be independent standard Brownian motions.

Let $\Psi_{N,t}$ be a solution to the $N$-particle equation \eqref{eqmainNpartBel}
with $H(N)$ of type either \eqref{eqHambinaryinter} or \eqref{eqHambinaryinter1}, with some initial condition
$\Psi_{N,0}$, $\|\Psi_{N,0}\|_2=1$. Let $\psi_{j,t}$ be solutions to equations \eqref{eqmainnonlinpartBel} with
the identical initial conditions
$\psi_{j,0}=\psi\in L^2(X)$, $\|\psi_{j,0}\|=1$.

Then
\begin{equation}
\label{eq2thmynonlinSch}
\E \al_N(t) \le e^{7t\|A\|_{HS}}\al_{N}(0)+(e^{7t\|A\|_{HS}}-1)\frac{1}{\sqrt N},
\end{equation}
and
\begin{equation}
\label{eq3thmynonlinSch}
\E \al_N(t) \le e^{7t\|V\|_{\infty}}\left(\al_{N}(0)+\frac{1}{N}\right)
+\frac{2\|V\|_{\infty}}{\sqrt N} \int_0^t e^{7(t-s)\|V\|_{\infty}}
 \int_X \sqrt{\E (|\psi_{j,s}(z)|^4)} \, dz\, ds,
\end{equation}
in case of the integral operator and the multiplication operator of interaction respectively.
\end{theorem}

The proof will be given in Section \ref{secproofconv}.
Let us make some comments. The assumption that $H$ and $L$ are bounded is not essential, and was made only to
simplify the presentation. In fact, as seen from the proof, neither $H$ nor $L$ enter any essential calculations or bounds,
so that to include unbounded $H$ and $L$ one just has to carefully describe all domains. The assumption of boundedness of $V$ can be
essentially relaxed, for instance the assumption of the deterministic case developed in \cite{KnowlesPickl}, that is,
$V\in L^2(\R^d)+L_{\infty}(\R^d)$ should work.

Unlike simple estimate \eqref{eq2thmynonlinSch},
 application of \eqref{eq3thmynonlinSch} requires some additional estimates of the r.h.s.,
 which we are not dealing with here,
 paying the main attention to the integral type interactions.

Of course, everything remains unchanged for a vector-valued $L$.

Our key restrictive assumption is the conservativity $L^*=-L$. By-passing it would require certain additional ideas.

For the application to the control theory some further extension is needed.
Let us formulate this result (proof will be given in Section \ref{secproofconvco})
  for the case of the integral interaction only. Let us assume that
the individual Hamiltonian $H$ has a control component, that is, it can be written as $H+u\hat H$ with two self-adjoint
operators $H$ and $\hat H$ and $u$ a real control parameter. Suppose that, for the limiting evolution, $u$ is chosen as a
 certain function of an observed density matrix $\ga_{j,t}$, that is $u=u(t,\ga_{j,t})$, while in the approximating $N$
 particle evolution one chooses $u$ based on the approximation $\Ga^{(j)}_{N,t}$ of $\ga_{j,t}$, that is as
 $u=u(t,\Ga^{(j)}_{N,t})$, where  $\Ga^{(j)}_{N,t}$ denotes the partial trace of $\Ga_{N,t}$ over all variables except for
 the $j$th. Thus the $N$ particle evolution \eqref{eqmainNpartBeldens} generalizes to the following nonlinear evolution
 (where we again omit the index $t$):
 \[
d \Ga_N
=-i\sum_j [H_j+u(t,\Ga^{(j)}_N) \hat H_j, \Ga_N] -\frac{i}{N} \sum_{l<j\le N} [A_{lj},\Ga_N]
\]
\begin{equation}
\label{eqmainNpartBeldensco}
+\sum_j (L_j\Ga_N L_j^* -\frac12 L^*_jL_j \Ga_N -\frac12 \Ga_N L^*_jL_j)\, dt
+\sum_j (\Ga_N L^*_j+L_j \Ga_N) dY_j,
\end{equation}
And the equation \eqref{eqmainnonlinpartBel1} generalizes to the following equation
 \begin{equation}
\label{eqmainnonlinpartBel1co}
d\psi_{j,t}(x) =-[i(H_j+u(t,\ga_{j,t})\hat H_j +A^{\bar \eta_t}) \psi_{j,t}(x)
+\frac12  L^*L \psi_{j,t}(x) ]\,dt+L_j\psi_{j,t}(x) \, dY^j_t.
\end{equation}

\begin{theorem}
\label{thmynonlinSchco}
Under all the assumptions of Theorem \ref{thmynonlinSch}, but assuming evolutions
\eqref{eqmainNpartBeldensco} and \eqref{eqmainnonlinpartBel1co} instead of
\eqref{eqmainNpartBeldens} and \eqref{eqmainnonlinpartBel1} respectively,
and assuming that the function $u(\ga)$ is Lipschitz in the sense that
\begin{equation}
\label{eq1thmynonlinSchco}
|u(t,\ga)-u(t,\tilde \ga)|\le \ka \, {\tr} |\ga -\tilde \ga|,
\end{equation}
it follows that the estimate \eqref{eq2thmynonlinSch} generalizes to the estimate

\begin{equation}
\label{eq2thmynonlinSchco}
\E \al_N(t) \le \exp\{ 7(\|A\|_{HS}+\ka \|\hat H\|)t\} \al_{N}(0)
+(\exp\{7(\|A\|_{HS}+\ka \|\hat H\|)t\}-1)\frac{1}{\sqrt N}.
\end{equation}
\end{theorem}

The comments related to the previous theorem remain valid.
However the boundedness of the control part $\hat H$ of the control
Hamiltonian is essential (unlike the boundedness of $H$ and $L$):
its norm $\|\hat H\|$ explicitly enters the final estimate.

\section{Quantum MFG: main result}
\label{secmainres}


Let us consider the quantum dynamic game of $N$ players, where the dynamics
of the density matrix $\Ga_{N,t}$ is given by the controlled dynamics of type \eqref{eqmainNpartBeldensco},
though the control of each player can be chosen independently on the basis of its 'position' $\Ga^{(j)}_{N,t}$:

 \[
d \Ga_{N,t}
=-i\sum_j [H_j+u_j(t,\Ga^{(j)}_{N,t}) \hat H_j, \Ga_{N,t}] -\frac{i}{N} \sum_{l<j\le N} [A_{lj},\Ga_{N,t}]
\]
\begin{equation}
\label{eqmainNpartBeldenscoindiv}
+\sum_j (L_j\Ga_{N,t} L_j^* -\frac12 L^*_jL_j \Ga_{N,t} -\frac12 \Ga_{N,t} L^*_jL_j)\, dt
+\sum_j (\Ga_{N,t} L^*_j+L_j \Ga_{N,t}) dY^j_t.
\end{equation}

Assume further that control $u$ can be chosen from some bounded closed interval $U$ of the real line,
 that the initial matrix is the product of identical states,
\[
\Ga_{N,0}(x_1, \cdots, x_n; y_1, \cdots, y_N)=\prod_{j=1}^N \psi(x_j)\overline{\psi(y_j)},
\]
and that the payoff of each player on the interval $[t,T]$ is given by the expression
 \begin{equation}
\label{eqcostfunin}
P_j(t, W ; u(.)) =\int_t^T   \left( {\tr} (J_j \Ga_{N,s}) -\frac{c}{2} u_j^2(s)\right)\, ds +{\tr} (F_j\Ga_{N,T}),
\end{equation}
where $J$ and $F$ are some operators in $L^2(X)$ expressing the current and
the terminal costs of the agent, $J_j$ and $F_j$ denote their actions
on the $j$th variable, $c$ measures the cost of applying the control.

Notice for clarity that by the property of the partial trace, the payoff
\eqref{eqcostfunin} rewrites as
 \begin{equation}
\label{eqcostfunineq}
P_j(t, W ; u(.)) =\int_t^T  \left( {\tr} (J_j \Ga^{(j)}_{N,s}) -\frac{c}{2} u_j^2(s)\right)\, ds
+{\tr} (F_j\Ga^{(j)}_{N,T}),
\end{equation}
so that it really depends explicitly only on the individual partial traces $\Ga^{(j)}_{N,t}$,
which can be considered as quantum analogs of the positions of classical particles.

The limiting evolution of each player can be expected to be  described by the equations
\[
d\ga_{j,t}=-i[H+u_j(t,\ga_{j,t}) \hat H,\ga_{j,t}] \, dt-i[A^{\overline{\eta_t}}, \ga_{j,t}] \, dt
+(L\ga_{j,t} L^* -\frac12 L^*L \ga_{j,t} -\frac12 \ga_{j,t} L^*L)\, dt
\]
\begin{equation}
\label{eqmainnonlinpartBeldens}
+(\ga_{j,t} L^*+L \ga_{j,t}) dY^j_t, \quad \eta_t(x,y)=\lim_{N\to \infty}\frac{1}{N}\sum_{j=1}^N \ga_{j,t}(x,y),
\end{equation}
with payoffs given by
 \begin{equation}
\label{eqcostfuninlim}
P_j(t, W ; u(.)) =\int_t^T  \left(  {\tr} (J \ga_{j,s}) -\frac{c}{2} u_j^2(s)\right)\, ds +{\tr} (F\ga_{j,T}).
\end{equation}

\begin{remark} Let us stress that we are not stating that evolution \eqref{eqmainnonlinpartBeldens}
is in fact the limiting one for the $N$-agent quantum evolution in this general case,
because we really do not need it. We need it only in case when almost all players (actually except
for one only) are playing the same strategy $u_t^{com}$, in which case $\eta_t=\E \ga_{j,t}$ for all $j$ adhering
to the common strategy. In which sense this is the limiting evolution will be make explicit in the proof
of our main Theorem \ref{mainresult}.
\end{remark}

Let us say that the pair of functions $u^{MFG}_t(\ga)=u^{MFG}(t,\ga)$ with $t\in [0,T]$ 
and $\ga$ from the set of density matrices in $L^2(X)$, $u\in U$,
and $\eta^{MFG}_t(x,y)$ with $x,y\in X$, $t\in [0,T]$, solve the limiting MFG problem if
(i) $u_t(\ga)$ is an optimal feedback strategy for the stochastic control problem \eqref{eqmainnonlinpartBeldens},
\eqref{eqcostfuninlim} under the fixed function $\eta_t=\eta_t^{MFG}$ and (ii) $\eta_t^{MFG}$ arises
from the solution of \eqref{eqmainnonlinpartBeldens} under fixed $u_t=u_t^{MFG}$.

We can formulate it also in another equivalent way. For a function $u_t^{com}(\ga)$
(index 'com' from 'common') suppose we can solve
the Cauchy problem for SDE \eqref{eqmainnonlinpartBeldens} with $u_t=u_t^{com}$ defining the correlations
$\eta_t(x,y)=\E \ga_{j,t}(x,y)$. Given these correlations we may be able to find an optimal feedback
control for the individual control problem \eqref{eqmainnonlinpartBeldens},
\eqref{eqcostfuninlim} under the fixed function $\eta_t$
 defining the individually optimal feedback
control $u_t^{ind}(w)$ (index 'ind' from 'individual') from
equation \eqref{eqHJBIsforspecialhom2}. The main MFG consistency equation is then expressed by
the equation  $u_t^{com}=u_t^{ind}$. If it is fulfilled, the pair $u_t^{com}$ and $\eta_t$
solves the limiting MFG in the sense defined above.

Formulated in this way, the MFG problem is fully classical, though the state space is the
sphere in the Hilbert space (or the space of density matrices). In Section \ref{seclimforback}
we shall write down explicitly the corresponding classical forward backward system for
the case of finite-dimensional quantum mechanics (that is, for finite set $X$).
A bit new moment (but not very essential one) is that the initial conditions
are identical to all players meaning that the initial measure in the forward equation
is a Dirac atom  and thus has no density,
unlike what is usually assumed in MFG analysis. Therefore, the theorem on the existence of the solutions
(given below in Section \ref{seclimforback} for finite $X$)
 can be considered as an existence result for classical MFGs on manifolds.
 What makes this story truly quantum is the completely different
link with $N$-agent quantum game arising essentially from Theorem \ref{thmynonlinSchco}.
Expressed otherwise, we established the correspondence that to each quantum $N$-agent game assigns a classical
$N$-agent game on some Riemannian manifold (possibly infinite-dimensional), so that the limiting
MFG forward-backward system is identical for both the quantum game and its classical counterpart.

\begin{theorem}
\label{mainresult}
Let the conditions on $H,L,A$ from Theorem \ref{thmynonlinSch} hold and let $\hat H$ be a
bounded self-adjoint operator  in $L^2(X)$.
Assume that the pair  $u^{MFG}_t(\ga)$ and $\eta^{MFG}_t(x,y)$
solves the limiting MFG problem and moreover $u^{MFG}_t$ is Lipschitz in the sense of
inequality \eqref{eq1thmynonlinSchco}.
Then the strategies
\[
u_j(t,\Ga_{N_t})=u^{MFG}_t(\Ga^{(j)}_{N,t}),
\]
where $\Ga^{(j)}_{N,t}$ is the partial trace of $\Ga_{N,t}$ with respect to all variables
except of the $j$th, form a symmetric $\ep$-Nash equilibrium for the $N$-agent quantum game
described by \eqref{eqmainNpartBeldenscoindiv} and \eqref{eqcostfunin},
with $\ep$ being of order $N^{-1/4}$.
\end{theorem}

\begin{remark}
We assume the readers are familiar with the main concepts of game theory like Nash equilibrium, see any textbook
on game theory for instance \cite{KolMal}.
\end{remark}
A proof will be given in Section \ref{secproofexst}.

Let us briefly described an important extension concerning the information space of the players.
In the game above the players were allowed to have access to their individual partial traces $\Ga_{N,t}^{(j)}$.
In the spirit of classical MFGs one could imagine them to have access to 'empirical measures',
which in our case represent the average operators
 \begin{equation}
\label{eqdefquantempirmes}
\frac{1}{N} \sum_{j=1}^N \Ga_{N,t}^{(j)}
\end{equation}
considered as operators in $L^2(X)$.

\begin{remark}
Notice that this notion is very close, but different from the 'empirical measures' introduced in
\cite{GolsePaul}, where it is defined as the average in $L^2(X^n)$ of secondly quantized operators
evolved according to Heisenberg equations.
\end{remark}

Since the averages \eqref{eqdefquantempirmes} approach in expectation the expected correlations $\eta_t$,
allowing $u_j$   in \eqref{eqmainNpartBeldenscoindiv} to depend on this averages amounts
to the dependence on $\eta_t$ in the limit, which is already taken into account in the construction of the MFG
consistency problem. Consequently this additional information possibility will not change the result
of Theorem  \ref{mainresult}.

Let us also comment that the standard $\ep$-Nash equilibria concerns the result of one player deviating
from the common strategy. Here on can show quite similarly that even if a finite (but bounded) number 
of players deviate from the common MFG strategy, they cannot improve their payoff more than by an $\ep$.       

\section{Proof of Theorem \ref{thmynonlinSch}}
\label{secproofconv}

We shall use the notations from Section \ref{secmynonlinSchrod} without further reminder.
For definiteness, we shall give the argument for the case of the integral operator $A$ of interaction,
noting occasionally some specific features of another case.

{\it Step 1.}
In case $L^*=-L$ equation \eqref{eqmainNpartBeldens} for the density matrix of $N$
particle evolution takes the form
\[
d \Ga_{N,t}
=-i\sum_j [H_j,\Ga_{N,t}]-\frac{i}{N} \sum_{l<j\le N} [A_{lj},\ga_{N,t}]
\]
\begin{equation}
\label{eqmainNpartBeldensrep}
+\sum_j (\frac12 L_jL_j \Ga_{N,t} +\frac12 \Ga_{N,t} L_jL_j-L_j\Ga_{N,t} L_j )\, dt
+\sum_j [L_j, \Ga_{N,t}] dY_j,
\end{equation}
or with $V_{lj}$ instead of  $A_{lj}$  in case of an integral operator of interaction.

Our first objective is to calculate the differential $d\al_{N,j}$ for $\al_{N,j}$
defined by \eqref{eqforalpha}.

One of the key property of the conservative case is the preservation of the trace, so that
${\tr} \, \Ga_{N,t}=1$ and  ${\tr} \, \ga_{j,t}=\|\psi_{j,t}\|^2=1$ for all $t$ almost surely, because it
was assumed that  ${\tr} \, \Ga_{N,0}=1$ and ${\tr} \,\ga_{j,0}=\|\psi_{j,0}\|^2=1$.
Hence the operators $q_{j,t}=\1-\ga_{j,t}$ are orthogonal projectors in $L^2(X)$, which are also identified
with the orthogonal projectors in $L^2(X^N)$ by making them act on the $j$th variable.

Important role in calculations belongs to the averaging operator
\[
\hat m_N(t)=\frac{1}{N} \sum_{j=1}^N q_{j,t},
\]
in $L^2(X^N)$. In terms of these operators one can  rewrite \eqref{eqforalpha}
in the following equivalent form
\begin{equation}
\label{eqforalphacon}
\al_{N,j}(t)=1-{\tr}(\ga_{j,t} \Ga_{N,t})={\tr}(q_{j,t} \Ga_{N,t}),
\end{equation}
so that (by the i.i.d. property of $q_{j,t}$)
\begin{equation}
\label{eqforalphaconex}
\E \al_{N,j}(t)=\E \, {\tr}(q_{j,t} \Ga_{N,t})=\E \, {\tr}(m_N(t) \Ga_{N,t}).
\end{equation}

Using definition \eqref{eqforalpha}  and Ito's product rule we derive that
\begin{equation}
\label{eqforalphaconder}
d \al_{N,j}(t)=-{\tr} (d\Ga_{N,t} \ga_{j,t})-{\tr} (\Ga_{N,t} \, d \ga_{j,t})-{\tr} (d\Ga_{N,t} \, d \ga_{j,t}).
\end{equation}
The first nice observation is that the stochastic part (the terms with differentials $dY^j_t$)
vanishes in this expression. In fact, it equals
 \[
- {\tr} \left(\sum_{k=1}^N [L_k, \Ga_{N,t}] \ga_{j,t} \, dY^k_t+\Ga_{N,t} [L_j, \ga_{j,t}] dY^j_t \right).
\]
All terms with $k\neq j$ vanish, because of the commutativity of the trace and because $L_k$ and $\ga_{j,t}$ commute.
Thus this stochastic part reduces to
\[
- {\tr} (-\Ga_{N,t} L_j \ga_{j,t}+L_j \Ga_{N,t} \ga_{j,t}-\Ga_{N,t} \ga_{j,t} L_j+\Ga_{N,t} L_j \ga_{j,t}) dY^j_t=0.
\]
Therefore we can write further $\dot \al_{N,j}(t)$ instead of $d \al_{N,j}(t)$.
Next remarkable fact is that the operators $L_j$ cancel completely from the expression for $\dot \al_{N,j}(t)$.
In fact, as follows from \eqref{eqforalphaconder} and Ito's rule, their contribution to $\dot \al_{N,j}(t)$
equals
\[
-{\tr} \left(\sum_k (\frac12 L^2_k \Ga_{N,t} +\frac12 \Ga_{N,t} L_k^2-L_k \Ga_{N,t} L_k)\ga_{j,t}\right)
\]
\[
-{\tr}\left(\Ga_{N,t} (\frac12 L^2_j \ga_{j,t} +\frac12 \ga_{j,t} L_j^2-L_j \ga_{j,t} L_j)\right)
-{\tr} ([L_j, \Ga_{N,t}][L_j,\ga_{j,t}]).
\]
Again terms with $k\neq j$ vanish and thus this expression reduces to
\[
-{\tr} \left( (\frac12 L^2_j \Ga_{N,t} +\frac12 \Ga_{N,t} L_j^2-L_j \Ga_{N,t} L_j)\ga_{j,t}
+\Ga_{N,t} (\frac12 L^2_j \ga_{j,t} +\frac12 \ga_{j,t} L_j^2-L_j \ga_{j,t} L_j)
+[L_j, \Ga_{N,t}][L_j,\ga_{j,t}]\right)
\]
\[
=-{\tr}(\Ga_{N,t} ( L^2_j \ga_{j,t} +\ga_{j,t} L_j^2-2L_j \ga_{j,t} L_j+[L_j,\ga_{j,t}]L_j-L_j[L_j,\ga_{j,t}]))=0.
\]
Thus, denoting $A_j^{\bar \eta_t}$ the operator $A^{\bar \eta_t}$ acting on the $j$th variable, we obtain
\[
\dot \al_{N,j}(t)
=i \, {\tr} ([H_j+A_j^{\bar \eta_t},\ga_{j,t}]\ga_{N,t})
 +i \, {\tr} (\ga_{j,t} [H(N),\ga_{N,t}])
 \]
 \[
 =i \, {\tr} ([H_j+A_j^{\bar \eta_t},\ga_{j,t}]\Ga_{N,t})
 +i \, {\tr} ([\ga_{j,t},H(N)] \Ga_{N,t})
 \]
 \[
 =-i \, {\tr} ([H_j+A_j^{\bar \eta_t},q_{j,t}]\Ga_{N,t})
 +i \, {\tr} ([H(N),q_{j,t}] \Ga_{N,t})
 \]
 \[
 =i \, {\tr} ([H(N)-H_j-A_j^{\bar \eta_t},q_{j,t}] \Ga_{N,t})
\]
\begin{equation}
\label{eqderalin}
=i \, {\tr} ([\frac{1}{N}\sum_{m\neq j} A_{mj}-A_j^{\bar \eta_t},q_{j,t}] \ga_{N,t}).
\end{equation}

Since ${\tr}(ABC)=\overline{{\tr}(BCA)}$ for any self-adjoint operators $A,B,C$, it follows
\[
|\dot \al_{N,j}(t)|
\le 2|{\tr} \left((\frac{1}{N}\sum_{m\neq j} A_{mj}-A_j^{\bar \eta_t})q_{j,t} \ga_{N,t}\right)|
\]
\begin{equation}
\label{eqforalphaconder1}
\le \frac{2}{N} \, |{\tr} \left((\sum_{m\neq j} A_{mj}-(N-1)A_j^{\bar \eta_t}) q_{j,t} \ga_{N,t}\right)|
+\frac{2}{N} \, |{\tr} ( A_j^{\bar \eta_t} q_{j,t} \ga_{N,t})|.
\end{equation}

{\it Step 2.}
The main observation allowing to achieve some cancellation in the first term
of \eqref{eqforalphaconder1} is the following equation:
\begin{equation}
\label{eqmaincancel}
\ga_{m,t}A_{jm}\ga_{m,t} =\ga_{m,t}A_j^{\overline{\ga_{m,t}}}.
\end{equation}

This is proved by inspection. In fact, the operator $\ga_{m,t}A_{jm}\ga_{m,t}$ acts as
\[
\ga_{m,t}A_{jm}\ga_{m,t} f(x,y)
=\int_{X^4} \psi_{m,t}(y)\bar \psi_{m,t} (z) A(x,z;x',w)\psi_{m,t}(w)\bar \psi_{m,t} (y') f(x',y') dx'dy'dz dw
\]
(where $x,y$ denote the variables $(x_j, x_m)$ of the collection $(x_1, \cdots, x_N)$),
and hence has the kernel
\[
\int_{X^2} \psi_{m,t}(y)\bar \psi_{m,t} (z)  A(x,z;x',w)\psi_{m,t}(w)\bar \psi_{m,t} (y') dz dw.
\]
The operator $\ga_{m,t}A^{\eta}_j$ acts on $f(x,y)$ as
\[
(\ga_{m,t}A^{\eta}_j f)(x,y)= \int_{X^4} \psi_{m,t}(y) \bar \psi_{m,t}(y') A(x,z;x',w)\eta(z,w) \psi(x',y') dz dw dx' dy',
\]
and \eqref{eqmaincancel} follows.

In case of the multiplication operator of interaction, formulas \eqref{eqforalphaconder1}
and \eqref{eqmaincancel} remain valid with the operators
$V_{ml}$ and $V_j^{\xi_t}$ instead of the operators $A_{mj}$ and $A_j^{\bar \eta_t}$.
In fact for the case of multiplication operator this formula was introduced and exploited in
 \cite{KnowlesPickl}.

{\it Step 3.}

Let us introduce the random function
\[
\de_{N,t}(z,w) =  \frac{1}{N-1}\sum_{m\neq j} \bar \psi_{j,t}(z) \psi_{j,t}(w)-\overline{\eta(z,w)}
=\frac{1}{N-1}\sum_{m\neq j}  \overline{\ga_{j,t}(z,w)}-\overline{\eta(z,w)}.
\]
By the law of large numbers $\de$ tends to $0$, as $N\to \infty$ for any $j$.
More precisely, $\E \de_{N,t}=0$ and
  \begin{equation}
\label{estimatedelta}
\E |\de_{N,t}(z,w)|^2=\text{Var} \, (\de_{N,t}(z,w)) \le \frac{1}{N-1} \text{Var} (\ga_{j,t}(z,w)).
\end{equation}
We can write
 \[
 A_j^{\bar \eta_t}=\frac{1}{N-1}\sum_{m\neq j}A_j^{\overline{\ga_{m,t}}}-A_j^{\de_{N,t}}
 \]
 and therefore
 \begin{equation}
\label{eqforalphaconder2}
\dot \al_{N,j}(t)
=\frac{i}{N} \, {\tr} ([\sum_{m\neq j} (A_{mj}-A_j^{\overline{\ga_{m,t}}}),q_{j,t}] \ga_{N,t})
+\frac{i(N-1)}{N}\, {\tr} ([A^{\de_{N,t}},q_{j,t}] \ga_{N,t})
-\frac{i}{N} \, {\tr} ([A_j^{\bar \eta_t},q_{j,t}] \ga_{N,t}).
\end{equation}
Consequently,
 \begin{equation}
\label{eqforalphaconder3}
\E |\dot \al_N(t)| \le 2\E (I+II+III),
\end{equation}
with
\[
I= \frac{1}{N}  |{\tr} \left((\sum_{m\neq j} (A_{mj}-A_j^{\overline{\ga_{m,t}}})q_{j,t} \ga_{N,t}\right)|,
\]
\[
II=|{\tr} \left([A^{\de_{N,t}} q_{j,t} \ga_{N,t}\right)|, \quad
III=\frac{1}{N} |{\tr} ( A_j^{\bar \eta_t} q_{j,t} \ga_{N,t})|.
\]

In case of the multiplication operator of interaction we define
\[
\de_{N,t}(z) =  \frac{1}{N-1}\sum_{m\neq j} (|\psi_{j,t}(z)|^2-\xi_t(z)),
\]
so that $\E\de_{N,t}(z) =0$ and
\[
\text{Var} \, (\de_{N,t}(z))=\E(|\de_{N,t}(z)|^2)=\E (|\psi_{j,t}(z)|^4)/(N-1).
\]
{\it Step 4.}

Since $|{\tr}(BC)|\le {\tr}|B| \, \|C\|$ for any operators $B,C$, we have
\[
III \le \frac{1}{N} {\tr} |q_{j,t}\ga_{N,t}| \, \| A_j^{\bar \eta_t}\|
\le \frac{1}{N} \| A_j^{\bar \eta_t}\|.
\]

By the direct application of the Cauchy-Schwartz inequality, it follows that
\[
\| A_j^{\bar \eta_t}\| \le \|\eta\|_{L_2(X^2)} \|A\|_{HS}.
\]
But
\[
 \|\eta\|_{L_2(X^2)}={\tr} (\eta^2) \le 1
 \]
 (here $\eta^2$ means the square of $\eta$ as an operator in $L^2(X)$).
Consequently
\[
III \le \frac{1}{N}  \|A\|_{HS}.
\]

Similarly,
\[
II \le \|A\|_{HS} \|\de_{N,t}\|_{L^2(X^2)}
\]
Consequently, by \eqref{estimatedelta},
\[
\E \, II \le  \sqrt {\E (II^2)} \le \|A\|_{HS} (\E \|\de_{N,t}\|^2_{L^2(X^2)})^{1/2}
= \|A\|_{HS} \int_{X^2} (\E |\de_{N,t}(z,w)|^2 dzdw)^{1/2}
\]
\[
 \le \frac{1}{\sqrt{N-1}}\|A\|_{HS} \left(\int_{X^2}  \text{Var} (\ga_{j,t}(z,w)) dzdw\right)^{1/2}.
 \]
Since
\[
\int_{X^2}  \text{Var} (\ga_{j,t}(z,w)) dzdw \le \int_{X^2} \E |\ga_{j,t}(z,w)|^2 dzdw
\]
\[
=\E \int_{X^2} |\ga_{j,t}(z,w)|^2 dzdw =\E \, {\tr} \, \ga_{j,t}^2,
\]
where in the last term $\ga_{j,t}^2$ is the square of $\ga_{j,t}$ as an operator in $L^2(X)$,
it follows that
\[
\int_{X^2}  \text{Var} (\ga_{j,t}(z,w)) dzdw
\le \E ({\tr} (\ga_{j,t}) \|\ga_{j,t}\|) \le 1.
\]
Consequently,
\[
\E \, II \le \frac{1}{\sqrt{N-1}}\|A\|_{HS}.
 \]

In case of the multiplication operator of interaction we obtain
\[
III \le \frac{1}{N}  \|V\|_{\infty}, \quad II \le \|V\|_{\infty} \int_X |\de_{N,t}(z)| dz,
\]
and
\[
\E \, II \le \|V\|_{\infty} \int_X \E |\de_{N,t}(z)| dz
\le  \|V\|_{\infty} \int_X (\E |\de_{N,t}(z)|^2)^{1/2} \, dz
\]
 \begin{equation}
\label{eqforalphaconder4}
\le  \frac{1}{\sqrt{N-1}} \|V\|_{\infty} \int_X [\E( |\psi_{j,t}(z)|^4)]^{1/2} \, dz.
\end{equation}
Thus the estimate in this case becomes more involved.

{\it Step 5.}

Dealing with $I$ we plan to use the cancellation formula  \eqref{eqmaincancel}. To this end, we write
\[
I=\frac{1}{N}  |(\Psi_{N,t}, \sum_{m\neq j} (A_{mj}-A_j^{\overline{\ga_{m,t}}}) q_{j,t} \Psi_{N,t})|
\]
\[
\le \frac{1}{N}  \sum_{m\neq j}  |(\Psi_{N,t}, (q_{m,t}+\ga_{m,t})(A_{mj}-A_j^{\overline{\ga_{m,t}}}) (q_{m,t}+\ga_{m,t})q_{j,t} \Psi_{N,t})|.
\]
By \eqref{eqmaincancel}, the term containing two multipliers $\ga_{m,t}$ vanishes, so that
$I\le I_1+I_2$ with
 \[
I_1= \frac{1}{N}  \sum_{m\neq j}  |(\Psi_{N,t}, q_{m,t} (A_{mj}-A_j^{\overline{\ga_{m,t}}}) q_{j,t} \Psi_{N,t})|,
 \]
 \[
I_2 =\frac{1}{N}  \sum_{m\neq j}  |(\Psi_{N,t}, \ga_{m,t})(A_{mj}-A_j^{\overline{\ga_{m,t}}}) q_{m,t} q_{j,t} \Psi_{N,t})|.
 \]
 For the first term we get the estimate
 \[
I_1 \le \frac{1}{N} \sum_{m\neq j}
\| q_{m,t} \Psi_{N,t}\| \, \| q_{j,t} \Psi_{N,t}\| \, \|A_{mj}-A_j^{\overline{\ga_{m,t}}}\|
 \]
 \[
 \le \frac{2}{N} \sum_{m\neq j}  \| q_{m,t} \Psi_{N,t}\| \, \| q_{j,t} \Psi_{N,t}\| \|A\|_{HS}
 \le \frac{2}{N}  \|A\|_{HS} \sum_{m\neq j}  (\al_{N,m}(t)\al_{N,j}(t))^{1/2}.
 \]
 Consequently
 \[
  \E I_1 \le 2 \|A\|_{HS} \E \, \al_N(t).
  \]
With $I_2$ we repeat the transformation above writing
\[
I_2\le\frac{1}{N}  \sum_{m\neq j}
|(\Psi_{N,t}, q_{j,t}\ga_{m,t})(A_{mj}-A_j^{\overline{\ga_{m,t}}}) q_{m,t} q_{j,t} \Psi_{N,t})|
\]
\[
+\frac{1}{N}  \sum_{m\neq j}
|(\Psi_{N,t}, \ga_{j,t}\ga_{m,t})(A_{mj}-A_j^{\overline{\ga_{m,t}}}) q_{m,t} q_{j,t} \Psi_{N,t})|.
\]
The first term is estimated as $I_1$ above and in the second term the operator $A_j^{\overline{\ga_{m,t}}}$ cancels,
since it commutes with $q_{m,t}$. Thus we obtain that
\[
\E I_2 \le  2 \|A\|_{HS} \E \al_N(t)
+\frac{1}{N}  \E \sum_{m\neq j}  |(\Psi_{N,t}, \ga_{j,t}\ga_{m,t}A_{mj} q_{m,t} q_{j,t}\Psi_{N,t})|.
\]

The second term here is exactly the same as the one that appears in both \cite{Pickl} and \cite{KnowlesPickl},
where nice specific tricks were invented
(based on the inversion of operator $\hat m_N(t)$) to deal with it. Let us note reproduce
these arguments (see e.g. p. 116 of \cite{KnowlesPickl}), but just give the resulting estimate:

\[
\E I_2 \le  2 \|A\|_{HS} (\E \al_N(t)+\sqrt 2(\E \al_{N,t}+\frac{1}{N})).
\]

In case of the multiplication operator of interaction we obtain
all estimates in Step 5 remain valid with $\|V\|_{\infty}$ instead of $\|A\|_{HS}$.

{\it Step 6.} Putting all estimates above together we get (for $N>1$) that
 \begin{equation}
\label{eqforalphaconder5}
\E|\dot \al_N(t)|\le 7 \|A\|_{HS} \left(\E \, \al_N(t)+\frac{1}{\sqrt N}\right),
\end{equation}
and
\[
\E|\dot \al_N(t)|\le 7 \|V\|_{\infty} \left(\E \, \al_N(t)+\frac{1}{N}\right)
+\frac{2}{\sqrt N} \|V\|_{\infty}  \int_X \sqrt{\E(|\psi_{j,t}(z)|^4)} \, dz,
\]
for the case of the integral operator and the multiplication operator of interaction respectively.

Applying Gronwall's lemma yields \eqref{eq2thmynonlinSch} and \eqref{eq3thmynonlinSch}.

\section{Proof of Theorem \ref{thmynonlinSchco}}
\label{secproofconvco}

Following the lines of the proof of Theorem \ref{thmynonlinSch},
we arrive to the equation

\begin{equation}
\label{eqderalincon}
\dot \al_{N,j}(t)=i \, {\tr} ([\frac{1}{N}\sum_{m\neq j} A_{mj}-A_j^{\bar \eta_t}
+(u(t,\Ga^{(j)}_{N,t})-u(t,\ga_{j,t})) \hat H,q_{j,t}] \ga_{N,t}),
\end{equation}
generalizing equation \eqref{eqderalincon}.
Thus we need to get the estimate for the term
\[
 |{\tr} [((u(t,\Ga^{(j)}_{N,t})-u(t,\ga_{j,t})) \hat H,q_{j,t}] \ga_{N,t})|
 \le 2  |{\tr} (u(t,\Ga^{(j)}_{N,t})-u(t,\ga_{j,t}) \hat H q_{j,t} \ga_{N,t})|
 \]
 \[
 \le 2 | u(t,\Ga^{(j)}_{N,t})-u(t,\ga_{j,t})| \|\hat H\| \sqrt \al_{N,j}(t).
 \]
 By \eqref{eq1thmynonlinSchco},
 \[
 | u(t,\Ga^{(j)}_{N,t})-u(t,\ga_{j,t})|\le \ka \, {\tr} |\Ga^{(j)}_{N,t}-\ga_{j,t}|.
 \]
 By \eqref{ineqKnPi},
 \[
  {\tr} |\Ga^{(j)}_{N,t}-\ga_{j,t}| \le 2 \sqrt{2 \al_{N,j}(t)}.
  \]
  Therefore,
  \[
  |{\tr} (u(\Ga^{(j)}_{N,t})-u(\ga_{j,t}) \hat H,q_{j,t}] \ga_{N,t})|
  \le 4\sqrt 2 \ka \|\hat H\|    \al_{N,j}(t),
  \]
  and thus
  \[
  \E |{\tr} (u(\Ga^{(j)}_{N,t})-u(\ga_{j,t}) \hat H,q_{j,t}] \ga_{N,t})|
  \le  4\sqrt 2 \ka \|\hat H\|   \E \al_N(t).
  \]
  Adding this term to the r.h.s. of \eqref{eqforalphaconder5} and using again Gronwall's lemma
  yields \eqref{eq2thmynonlinSchco}.

\section{Proof of Theorem \ref{mainresult}}
\label{secproofexst}

 Assume that all players, except for one of them,
say the first one, are playing according to the MFG strategy $u^{MFG}(t,\Ga_{N,t}^{(j)})$, $j>1$, and
the first player is following some other strategy  $\tilde u(t,\Ga_{N,t}^{(1)})$.
By the law of large numbers (which is not affected by a single deviation),
all $\eta_t^j$ are equal and are given by the formula  $\eta_t=\E \ga_{j,t}$
for all $j>1$. Moreover, $\E \al_{N,j}(t)=\E \al_N(t)$ are the same for all $j>1$.

Following the proof of Theorem \ref{thmynonlinSchco} we obtain
\begin{equation}
\label{eq1mainresult}
\dot \al_{N,j}(t)=i \, {\tr} ([\frac{1}{N}\sum_{m\neq j} A_{mj}-A_j^{\bar \eta_t}
+(u^{MFG}(t,\Ga^{(j)}_{N,t})-u^{MFG}(t,\ga_{j,t})) \hat H,q_{j,t}] \ga_{N,t}),
\end{equation}
for all $j>1$. Up to an additive correction of magnitude not exceeding $4/N$
the r.h.s. can be substituted by the expression
\[
\dot \al_{N,j}(t)=i \, {\tr} ([\frac{1}{N}\sum_{m\neq j,1} A_{mj}-A_j^{\bar \eta_t}
+(u^{MFG}(t,\Ga^{(j)}_{N,t})-u^{MFG}(t,\ga_{j,t})) \hat H,q_{j,t}] \ga_{N,t}),
\]
which is then dealt with exactly as in the proof of Theorem \ref{thmynonlinSchco}
yielding the sam estimate  \eqref{eq2thmynonlinSchco} (with a corrected multiplier)
for $\E \al_N(t)=\E \al_{N,j}(t)$, $j>1$, that is
\begin{equation}
\label{eq2mainresult}
\E \al_N(t) \le \exp\{ 7(\|A\|_{HS}+\ka \|\hat H\|)t\} \al_{N}(0)
+(\exp\{7(\|A\|_{HS}+\ka \|\hat H\|)t\}-1)\frac{5}{\sqrt N}.
\end{equation}

Similarly the same estimate is obtained for $\E \al_{N,1}(t)$.
Since our initial condition are supposed to be product of identical functions,
the initial $\al_{N,j}(0)$ vanish yielding
\[
\E \al_{N,j}(t) \le C(T) N^{-1/2}
\]
for all $j$ and a constant $C(T)$ depending on $\|A\|_{HS},\ka, \|\hat H\|$.

  We can now compare the expected payoffs \eqref{eqcostfunineq} received by the
  players in the $N$-player quantum game
  with the expected payoff \eqref{eqcostfuninlim} received in the limiting game.
  For each $j$th player the difference is bounded by
  \[
\E \int_t^T  |{\tr} (J (\Ga^{(j)}_{N,s}-\ga_{j,s}))|\, ds
+\E |{\tr} (F(\Ga^{(j)}_{N,T}-\ga_{j,T}))|.
\]
Since,
\[
|{\tr} (J (\Ga^{(j)}_{N,s}-\ga_{j,s}))|\le \|J\| {\tr} |\Ga^{(j)}_{N,s}-\ga_{j,s}|,
\]
and by \eqref{ineqKnPi},
\[
 {\tr} |\Ga^{(j)}_{N,s}-\ga_{j,s}|\le 2\sqrt{2\al_{N,j}(s)},
 \]
 it follows that the expectation of the difference of the payoffs is bounded by
 \[
 2\sqrt 2 (\|J\|T +\|F\|) \sup_t \E \sqrt{\al_{N,j}(t)}
\]
\[
 \le 2\sqrt 2 (\|J\|T +\|F\|) \sup_t \sqrt{\E  \al_{N,j}(t)}
 \le (\|J\|T +\|F\|) C(T) N^{-1/4},
 \]
with a constant $C(T)$ depending on $\|A\|_{HS},\ka, \|\hat H\|$.

But by the assumption of the Theorem, $u_t^{MFG}$ is the optimal choice
for the limiting optimization problem. Hence the claim of the theorem follows.

 \section{Limiting MFG problem in finite-dimensional case}
\label{seclimforback}

Here we derive some existence results for the limiting MFG problem to finite-dimensional
 quantum systems, referred to as atoms. This section is essentially independent of other 
 material (though we use the notations introduced previously) and can be also considered 
 as an introduction to classical MFGs with a drift control on compact Riemannian manifolds, 
 as exemplified by the complex projective space.

The state space of each atom is a finite-dimensional Hilbert space $\C^{n+1}$.
The interaction will be given by the tensor $A(j,k;j',k')$, $j,k,j',k' \in \{0, \cdots, n\}$ such that
\[
A(j,k;j',k')=A(k,j;k',j'), \quad A(j,k;j',k')=\overline{A(j',k';j,k)}.
\]
The case of a multiplication operator of interaction is now fully included, as it corresponds
 to the diagonal tensor $A$.

 \begin{remark}
 The typical physics choice of the interaction between qubits, that is for $n=1$, is the operator
 describing the possible exchange of photons,
 $A=a_1^*a_2+a_2^*a_1$, with the annihilation operators $a_1$ and $a_2$ of the two atoms.
 This interaction is given by the tensor $A(j,k;m,n)$ such that $A(1,0;0,1)=A(0,1;1,0)=1$ with other
 elements vanishing.
 \end{remark}

As in Section \ref{secmainres}, each atom is controlled by an agent that can control
a part of the individual Hamiltonians. Therefore the $N$-particle Hamiltonian \eqref{eqHambinaryinter}
will have the form
\begin{equation}
\label{eqHambinaryintercontfd}
 H(N)f(i_1, \cdots , i_N)=\sum_{j=1}^N (H_j+u_j \hat H_j)f(i_1, \cdots , i_N)
 + \frac{1}{N}\sum_{l<j\le N} A_{lj}f(i_1, \cdots , i_N),
\end{equation}
where $H$ and $\hat H$ are self-adjoint matrices in $\C^{n+1}$ and $H_j$, $\hat H_j$ denote
their actions on the variables $i_j$, $u_j$ is a control parameter of $j$th agent defining
the strength of the field (say electric or magnetic) described by the Hamiltonian $\hat H_j$.
For simplicity we assume that each $u_j$ can be chosen from some fixed interval $U=[-U_0,U_0]$.

Assuming that observations of each atom are performed by the coupling with the same
anti-Hermitian vector-valued operator $\L=(L^1,\cdots, L^K)$ and assuming that the initial
conditions and controls of the agents
are identical we expect from Theorem \ref{thmynonlinSch} that the limiting evolution,
as $N\to \infty$, will be described by the nonlinear equation \eqref{eqmainnonlinpartBel1},
that is

\begin{equation}
\label{eqmainnonlinpartBel1corep}
d\psi_{j,k,t} =-[i(H+u_j(t)\hat H)\psi_{j,k,t} +iA^{\bar \eta_t} \psi_{j,k,t}
-\frac12  \L^2 \psi_{j,k,t} ]\,dt+\L\psi_{j,k,t} \, dY^j_t,
\end{equation}
where
\begin{equation}
\label{eqmainnonlinpartBel1co1}
\L^2 =\sum_{p=1}^P (L^p)^2, \quad
\L \, dY^j_t=\sum_{p=1}^P L^p dY^{j,p}_t,
\end{equation}
with $Y^{j,p}_t$, $j=1, \cdots N$, $p=1, \cdots, P$, being independent standard Brownian motions.

Since quantum states are defined up to a complex multiplier, so that the state space is effectively
the complex projective space $\C P^n$, rather than the linear space $\C^{n+1}$, it is convenient to rewrite equation
\eqref{eqmainnonlinpartBel1co} in projective coordinates $w=(w_1,\cdots, w_n)$, with $w_k=w_{k,t}=\psi_{k,t}/\psi_{0,t}$
(where we omit index $j$ for brevity).
To shorten formulas it is also handy to use the $n+1$-dimensional vector
$W=(1,w_1,\cdots, w_n)$ with the additional coordinate $w_0=1$.
This rewriting is done by direct application of Ito's formula
(details of simple calculations given in paper \cite{KolDynQuGames})
yielding the following equation
 \[
dw_k =i[w_k((H+u\hat H+A^{\bar\eta})W)_0-((H+u\hat H+A^{\bar\eta})W)_k]\, dt
\]
\[
+\frac12 \sum_p [((L^p)^2W)_k- w_k((L^p)^2W)_0 ] \, dt
\]
 \begin{equation}
\label{eqqufiBlinnind3w}
+\sum_p[w_k(L^pW)^2_0 -(L^pW)_0 (L^pW)_k]\, dt
+\sum_p[(L^pW)_k -w(k)(L^pW)_0]\, dY^{j,p}_t.
\end{equation}

Recall that the coordinates $w=(w_1, \cdots, w_n)$ cover the open dense subset $V_0$ of $\C P^n$ arising
from the vectors $\psi=(\psi_0, \cdots, \psi_n)\in \C^{n+1}$ with $\psi_0\neq 0$.
The whole  $\C P^n$ is covered by $n+1$ such charts $V_j$, each describing the vectors with $\psi_j\neq 0$.

As was discovered in \cite{KolDynQuGames}, if one chooses as the coupling operators $L^p$ the $(n^2+2n)$ generalized
Gell-Mann matrices (in case of a qubit these are $3$ Pauli matrices), the third and fourth terms in
\eqref{eqqufiBlinnind3w} will vanish and the second order diffusion operator arising from the last term in
\eqref{eqqufiBlinnind3w} will coincide (up to a multiplier $2$) with the major (second order) part
of the Laplace-Beltrami operator on the corresponding complex projective space $\C P^n$.
In our case $L^p$ are assumed to be anti-Hermitian, rather than Hermitian operators in \cite{KolDynQuGames}.
However, as seen directly, multiplying all operators by the imaginary unit $i$
(turning Hermitian matrices to anti-Hermitian) does not affect this property. Hence,
choosing $L^p$ as the generalized Gell-Mann matrices multiplied by $i$,
equation   \eqref{eqqufiBlinnind3w} simplifies to
 \[
dw_k =i[w_k((H+u\hat H+A^{\bar\eta})W)_0-((H+u\hat H+A^{\bar\eta})W)_k]\, dt
\]
 \begin{equation}
\label{eqqufiBlinnind3wGM}
+\sum_p[(L^pW)_k -w_k(L^pW)_0]\, dY^{j,p}_t, \quad k=1, \cdots, n,
\end{equation}
and the second order diffusion operator arising from the last term of this equation
equals $2\De_{pro}$, where $\De_{pro}$  is the major (second order) part
of the Laplace-Beltrami operator on the corresponding complex projective space $\C P^n$.
This operator is both invariant (it looks the same in the projective coordinates of all $n+1$ charts
$V_j$ covering $\C P^n$) and non-degenerate,
which makes it the most handy for the analysis of optimal control.

In particular, for the most important case of a qubit, that is for $n=1$,
the following formula holds in real coordinates $x,y$ (with $w=x+iy$):
\begin{equation}
\label{eqquaqdrfromPau}
2\De_{pro} S(x,y)= \frac12 (1+x^2+y^2)^2\left(\frac{\pa ^2S}{\pa x^2}+\frac{\pa ^2S}{\pa y^2}\right),
\end{equation}
that is, the l.h.s. coincides exactly (up to a multiplier $2$) with the Laplace-Beltrami operator
on the $2$-dimensional sphere $S^2$.
In case of a qutrit, that is for $n=2$, the formula for $\De_{pro}$ on $\C P^2$ is as follows:
  \[
 \De_{pro}S(w_1,w_2)=(1+\sum_j|w_j|^2)
 \bigl[( 1+|w_1|^2)\frac{\pa^2S}{\pa w_1 \pa \bar w_1}
 +( 1+|w_2|^2))\frac{\pa^2S}{\pa w_2 \pa \bar w_2}
 \]
  \begin{equation}
\label{eqLaplRiemProjcom}
+w_1\bar w_2 \frac{\pa^2S}{\pa w_1 \pa \bar w_2}+\bar w_1 w_2 \frac{\pa^2S}{\pa \bar w_1 \pa w_2}
\bigr].
 \end{equation}

\begin{remark}
More generally, the same effect occurs (the diffusion operator
of evolution \eqref{eqqufiBlinnind3w} coincides with the second order term of the Laplace-Beltrami operator),
if one takes as $L^p$ an orthonormal basis of the Lie algebra
of the group of unitary matrices $U(n)$.
\end{remark}

To derive the MFG equation let us assume that some deterministic 'empirical measure' $\eta_t$ is given
and $u_j$ has to be chosen by $j$th agent to maximize the payoff
 \begin{equation}
\label{eqcostfun}
P(t, W ; u(.)) =\E \int_t^T    (\langle J\rangle_{W(s)} -\frac{c}{2} u^2(s))\, ds +\langle F\rangle_{W(T)},
\end{equation}
where $J=(J_{pm})$ and $F=(F_{pm})$ are some $(n+1)\times (n+1)$ matrices
expressing the current and the terminal costs of the agent,
$c$ measures the cost of applying the control and $\E$ denotes the expectation
with respect to the random trajectories $W(s)$ arising from dynamic \eqref{eqqufiBlinnind3wGM}
under the strategic choice of the control $u$.
Recall that
\[
\langle J\rangle_{W}=\frac{(W,JW)}{(W,W)}=\frac{\sum_{p,m} \overline{W_p}J_{pm}W_m}{\sum_m |W_m|^2}
=\frac{\sum_{p,m} \overline{W_p}J_{pm}W_m}{1+\sum_{m>0} |w_m|^2},
\]
with similar formula for $\langle F\rangle_W$.

The problem of dynamic maximization of \eqref{eqcostfun} is a standard problem of controlled diffusion,
and we can write down the standard Hamilton-Jacobi-Bellman (HJB) equation describing the optimal payoff $S$,
taking into account that the (uncontrolled) diffusion part is given by $2\De_{pro}$:

\[
0=\frac{\pa S_t}{\pa t}+2\De_{pro}S_t+\langle J\rangle_W+\sup_u \left\{u \, \Pi(\nabla S_t)-\frac{c}{2}u^2\right\}
\]
\[
+\sum_k {Re}[iw_k(H+A^{\bar \eta_t} W)_0-i(H+A^{\bar \eta_t} W)_k]\frac{\pa S_t}{\pa x_k}
\]
\begin{equation}
\label{eqHJBIsforspecialhom}
+ \sum_k {Im} [iw_k(H+A^{\bar \eta_t} W)_0-i( H+A^{\bar \eta_t} W)_k]\frac{\pa S_t}{\pa y_k}
\end{equation}
with
\begin{equation}
\label{eqHJBIsforspecialhom1}
\Pi(\nabla S)=\sum_k \left[ {Re}[iw_k(\hat H W)_0-i(\hat H W)_k]\frac{\pa S}{\pa x_k}
+ {Im} [iw_k(\hat H W)_0-i(\hat H W)_k]\frac{\pa S}{\pa y_k}\right].
\end{equation}
Here $\sup_u=\sup_{u\in [-U_0,U_0]}$ can be calculated explicitly yielding
\begin{equation}
\label{eqHJBIsforspecialhom2}
u=u_t(w,\nabla S_t(w))=\min [\max(U_0, \Pi(\nabla S_t)/c), -U_0].
\end{equation}

The optimal cost function $S_t$ is expected to satisfy the backward Cauchy problem for this equation,
that is, it is specified by the terminal condition
\[
S_T(W)=\langle F\rangle_{W(T)}.
\]

Next, for a given feedback control $u=u_t(w)$, the process \eqref{eqqufiBlinnind3wGM}
is a nondegenerate diffusion  in $\C P^n$ and consequently its distribution has density
(for $t>0$) with
respect to Lebesgue measure. In the coordinates $w=\{w_j=x_j+iy_j\}$, $j=1,\cdots,n$,
 of the chart $V_0$ this density $\mu_t(w)$
satisfies the forward Kolmogorov equation
\[
\frac{\pa \mu_t}{\pa t}=2\De_{pro}\mu_t
-\sum_k \frac{\pa}{\pa x_k}[{Re}(iw_k(H+u(w)\hat H+A^{\bar \eta} W)_0-i(H+u(w)\hat H+A^{\bar \eta} W)_k)\mu_t]
\]
\begin{equation}
\label{eqKolmogorovfor}
-\sum_k \frac{\pa}{\pa y_k} [ {Im}(iw_k((H+u(w)\hat H+A^{\bar \eta} W)_0-i((H+u(w)\hat H+A^{\bar \eta} W)_k)\mu_t].
\end{equation}
Since our process starts with some fixed initial state $\psi_0=(\psi_{00}, \cdots, \psi_{n0})$,
with the corresponding initial vector $w(0)$ such that $w(0)_k=\psi_{j,k,0}/\psi_{j,0,0}$, $k=1, \cdots, n$,
 the density $\mu_t$ satisfies the Dirac initial condition $\mu_0(dw)=\de(w-w(0))$.

As the control $u_t(\mu)$ is usually not a smooth function,
 we cannot expect to have strong solutions of either the HJB equation
\eqref{eqHJBIsforspecialhom} or the Kolmogorov equation \eqref{eqKolmogorovfor}.
Hence it is convenient to rewrite them in the so called mild (or integral) forms,
that already include the corresponding initial and terminal conditions,
that is as the equations
\[
S_t(w)=e^{2(T-t)\De_{pro} }S_T+\int_t^T e^{2(s-t)\De_{pro}}
(\langle J\rangle_{W}+\sup_u \left\{u \, \Pi(\nabla S_s)-\frac{c}{2}u^2\right\} ds
\]
\[
+\int_t^T e^{2(s-t)\De_{pro}}\sum_k  {Re}[iw_k(H+A^{\bar \eta_s} W)_0-i(H+A^{\bar \eta_s} W)_k]\frac{\pa S_s}{\pa x_k} \, ds
\]
\begin{equation}
\label{eqHJBIsforspecialhommild}
+\int_t^T e^{2(s-t)\De_{pro}}\sum_k {Im} [iw_k(H+A^{\bar \eta_s} W)_0-i( H+A^{\bar \eta_s} W)_k]\frac{\pa S_s}{\pa y_k}\, ds,
\end{equation}

and
\[
\mu_t(w)=e^{2t\De_{pro}}\mu_0(w)
\]
\[
-\int_0^t e^{2(t-s)\De_{pro}}
\sum_k \frac{\pa}{\pa x_k}[{Re}(iw_k(H+u(w)\hat H+A^{\bar \eta_s} W)_0-i(H+u(w)\hat H+A^{\bar \eta_s} W)_k)\mu_s] \, ds
\]
\begin{equation}
\label{eqKolmogorovformild}
-\int_0^t e^{2(t-s)\De_{pro}}
\sum_k \frac{\pa}{\pa y_k} [ {Im}(iw_k((H+u(w)\hat H+A^{\bar \eta_s} W)_0-i((H+u(w)\hat H+A^{\bar \eta_s} W)_k)\mu_s]ds,
\end{equation}
respectively.

\begin{remark} An alternative approach could be to use the viscosity solutions
to equations \eqref{eqHJBIsforspecialhom} and \eqref{eqKolmogorovfor}.
\end{remark}

The forward-backward system of MFG that express the consistency of the individual optimal control
and the dynamics is therefore the pair of equations \eqref{eqHJBIsforspecialhom} and \eqref{eqKolmogorovfor},
or more generally \eqref{eqHJBIsforspecialhommild} and \eqref{eqKolmogorovformild},
coupled via control \eqref{eqHJBIsforspecialhom2}. Namely, this system consists of the backward HGB
equation \eqref{eqHJBIsforspecialhom} with
\begin{equation}
\label{eqcorrelationsinfor}
\eta_t(k,l)=\E \ga_t (k,l)=\E (\psi_{k,t} \overline{\psi_{l,t}}), \quad k,l=0, \cdots, n,
\end{equation}
arising from equation \eqref{eqKolmogorovfor}, and the forward equation \eqref{eqKolmogorovfor}
with $u_t(w)$ arising from \eqref{eqHJBIsforspecialhom} via formula \eqref{eqHJBIsforspecialhom2}.
Notice also that $\psi$ is obtained by normalization from the vector $W$, so that
\begin{equation}
\label{eqcorrelationsinfor1}
\psi_{k,t}=\frac{W_{k,t}}{\sqrt{1+|w_t|^2}}.
\end{equation}

We can formulate it also in another equivalent way, as the consistency condition
$u_t^{com}(w)=u_t^{ind}(w)$, as in Section \ref{secmainres}.

Let us finally represent the MFG problem as a single anticipating equation.
First of all let $K(t,w,v)$ be the heat kernel related to the operator $2\De_{pro}$ on $\C P^n$,
that is, $K(t,v,w)$ is the solution of the corresponding heat equation
$(\pa K/\pa t)=2\De_{pro} K$ as a function of $(t>0,v\in \C P^n)$ and has the Dirac initial condition $K(0,v,w)=\de_w(v)$.
It is well known that $K(t,v,w)$ is a (infinitely) smooth function of $v,w$ for $t>0$
and that the Cauchy problem for this heat equation is well posed in $M$. Its resolving operators
 \begin{equation}
\label{eqheatsem}
e^{2t \De_{pro}} f(v)=\int_{\C P^n} K(t,v,w) f(w) \, d_nw
= \int_{\C^n} K(t,v,w) f(w) \, d_nw,
\end{equation}
 form a strongly continuous semigroup of contractions
in the space $C(\C P^n)$ of bounded continuous functions on $\C P^n$, equipped with the sup-norm.
Here $d_nw=\sqrt{\det g(w)}dw$ is the Riemannian volume on $\C P^n$,
where $dw=\prod dw_j=\prod (dx_j dy_j)$ and $g(w)$ denotes the
Riemannian metric on $\C P^n$ in coordinate $w$).
Notice that the chart $V_0$ (of vectors $\psi$ with $\psi_0\neq 0$) is isomorphic to $\C^n$
and covers $\C P^n$ up to a set of zero measure, so that integrating over $\C^n$ and $\C P^n$ in
\eqref{eqheatsem} is equivalent. Having in mind that the spaces $L_1\C P^n)$ of integrable
(with respect to $d_v$) functions is inserted into the space $\MC(\C P^n)$ of Borel measures
on $\C P^n$, it follows that the semigroup $S_t$ extends to the semigroup $e^{2t \De_{pro}}$
on  $\MC(\C P^n)$ that maps $\MC(\C P^n)\to L_1\C P^n)$ for any $t>0$ and acts according to the formula
 \begin{equation}
\label{eqheatsemmes}
e^{2t \De_{pro}} \mu(v)=\int_{\C P^n} K(t,v,w) \mu(dw).
\end{equation}
Let us note for clarity that when we consider this transformation on measures,
we identify the function $e^{2t \De_{pro}} \mu(v)$ on the l.h.s.
with the measure $e^{2t \De_{pro}} \mu(v)d_nv$ on $\C P^n$.

\begin{remark} As mentioned above, $\De_{pro}$ coincides with the Laplace-Beltrami operator
in case $n=1$. In case $n>1$ $\De_{pro}$ is in fact only the major second order part
of the Laplace-Beltrami operator defining the standard Brownian motion on $\C P^n$.
However, since they differ only by the bounded smooth first order parts, all well known asymptotic and
smoothness properties of the Laplace-Beltrami operator remain valid for $\De_{pro}$.
\end{remark}

Hence equation \eqref{eqKolmogorovformild} rewrites in the following form
\[
\mu_t(v)=\int_{\C^n} K(t,v,w)\mu_0(dw)
-\int_0^t \int_{\C^n} K(t-s,v,w) \, ds d_nw
\]
\[
\times \bigl(
\sum_k \frac{\pa}{\pa x_k}[{Re}(iw_k(H+u_s(w)\hat H+A^{\bar \eta_s} W)_0
-i(H+u_s(w)\hat H+A^{\bar \eta_s} W)_k)\mu_s(w)]
\]
\begin{equation}
\label{eqKolmogorovformild1}
+\sum_k \frac{\pa}{\pa y_k} [ {Im}(iw_k((H+u_s(w)\hat H+A^{\bar \eta_s} W)_0
-i((H+u_s(w)\hat H+A^{\bar \eta_s} W)_k)\mu_s(w)]\bigr).
\end{equation}
Let us reiterate that with some abuse of notation we identify measures
with their densities (with respect to Riemannian volume) so that $\mu(dw)=\mu(w) d_nw$.
Thus in \eqref{eqKolmogorovformild1} only $\mu_0(dw)$ denotes the initial Dirac measure (that has no density),
and all other $\mu_t(w)$ denote the densities.

Using integration by parts this rewrites as
\[
\mu_t(v)=\int_{\C^n} K(t,v,w)\mu_0(dw)
\]
\[
+\int_0^t \int_{\C^n} \sum_k \frac{\pa}{\pa x_k} [K(t-s,v,w) \sqrt{\det g(w)}] \, dw ds
\]
\[
\times [{Re}(iw_k(H+u_s(w)\hat H+A^{\bar \eta_s} W)_0-i(H+u_s(w)\hat H+A^{\bar \eta_s} W)_k)\mu_s(w)]
\]
\[
+\int_0^t \int_{\C^n} \sum_k \frac{\pa}{\pa y_k} [K(t-s,v,w)  \sqrt{\det g(w)}] \, dw ds
\]
\begin{equation}
\label{eqKolmogorovformild2}
\times [ {Im}(iw_k(H+u_s(w)\hat H+A^{\bar \eta_s} W)_0-i(H+u_s(w)\hat H+A^{\bar \eta_s} W)_k)\mu_s(w)].
\end{equation}

The advantage of this equation as compared to \eqref{eqKolmogorovformild1} and \eqref{eqKolmogorovfor}
is clear: no smoothness of the function $u(w)$ is required for this equation to make sense.

Next, let $u_t(w,\nabla S_t(w;\eta_{\ge t})$ be given by \eqref{eqHJBIsforspecialhom2}
and be considered as a functional of $w$ and the curve $\eta_{\ge t}$, by which we denote
the piece of curve $\eta_s$ for $s\in [t,T]$. Of course $S_t$ denotes here the solution
of \eqref{eqHJBIsforspecialhommild}.
Plugging this into the forward equation \eqref{eqKolmogorovformild2}
we get the following single nonlinear equation with the anticipating (depending on the future) r.h.s.:
\[
\mu_t(v)=\int_{\C^n} K(t,v,w)\mu_0(dw)
\]
\[
+\int_0^t \int_{\C^n} \sum_k \frac{\pa}{\pa x_k} [K(t-s,v,w) \sqrt{\det g(w)}] \, dw ds
\]
\[
\times [{Re}(iw_k(H+u_s(w,\nabla S_s(w;\eta_{\ge s}))\hat H+A^{\bar \eta_s} W)_0
-i(H+u_s(w,\nabla S_s(w;\eta_{\ge s}))\hat H+A^{\bar \eta_s} W)_k)\mu_s(w)]
\]
\[
+\int_0^t \int_{\C^n} \sum_k \frac{\pa}{\pa y_k} [K(t-s,v,w)  \sqrt{\det g(w)}] \, dw ds
\]
\begin{equation}
\label{eqMFGqu}
\times [ {Im}(iw_k((H+u_s(w,\nabla S_t(w;\eta_{\ge s}))\hat H+A^{\bar \eta_s} W)_0
-i((H+u_s(w,\nabla S_s(w;\eta_{\ge s}))\hat H+A^{\bar \eta_s} W)_k)\mu_s(w)],
\end{equation}
where
\begin{equation}
\label{eqMFGqu1}
\eta_t(k,m)=\E_{\mu_t} (\psi_{m,t} \overline{\psi_{m,t}})
=\int_{\C^n} (\psi_{m,t} \overline{\psi_{m,t}}) \, \mu_t(w) \, d_nw, \quad t>0,
\end{equation}
and
\[
\eta_0(k,m)=\psi_{m,0} \overline{\psi_{m,0}}
=\int_{\C^n} (\psi_{m,0} \overline{\psi_{m,0}}) \, \mu_0(dw).
\]
By \eqref{eqcorrelationsinfor1}, we also have
\[
\eta_t^{km}= \E_{\mu_t}\frac{W_{m,t} \overline{W_{m,t}}}{1+|w_t|^2}, \quad k,m=0, \cdots, n,
\]
with $|w_t|^2=\sum_{k=1}^n |w_{k,t}|^2$.

Now we formulate the global existence and local well-posedness for equation \eqref{eqMFGqu},
which is equivalent to the coupled system of equations
 \eqref{eqHJBIsforspecialhommild} and \eqref{eqKolmogorovformild}.

 \begin{theorem}
 \label{thonwelposedMFGman}
 (i) For sufficiently small $T$ equation \eqref{eqMFGqu} is well posed,
 that is for any $R>0$ there exist $T>0$ such that for all $S_T$ with $\|S_T\|_{C^1(\C P^n)} \le R$,
 equation  \eqref{eqMFGqu} has a unique solution.
 (ii) For any $T>0$ and any continuously differentiable function $S_T$ on $\C P^n$
 there exists a solution to equation \eqref{eqMFGqu}.
 \end{theorem}

\begin{proof}
As was mentioned this result can be considered as belonging to the theory of classical MFGs on manifolds.
Though being seemingly new (the author is unaware of any papers dealing with MFGs on manifolds), the
proof can be performed by the (by now) standard approach, of course enhanced by some
specific geometric analysis.  We will follow closely the method from
\cite{KoWe13}.

(i) We reformulate equation \eqref{eqMFGqu} as a fixed point problem in the following way.
Let $C_0([0,T], \MC(\C P^n))$ be the space of weakly continuous functions from $[0,T]$ to
$\MC(\C P^n)$, equipped with distance
\[
\|\mu^1_.-\mu^2_.\|_{\MC,T}=\sup_{t\in [0,T]}\sup_{\|f\|\le 1} |(f, \mu^1_t-\mu^2_t)|
\]
such that the initial point is fixed as the Dirac measure $\mu_0(dw)=\de(w-w(0))$.

\begin{remark} It is straightforward to see that this space is a complete metric space.
It is a bit nonstandard, because continuity in $t$ is defined as the weak one, while the
distance is defined via the strong Banach topology. The necessity to use such a hybrid
arises because curves starting with a Dirac measure and having densities
otherwise (that eventually would solve our problem) cannot
be strongly continuous.
\end{remark}

Equation \eqref{eqMFGqu} is the fixed point problem for the mapping $\Phi(\mu_.)$ in
$C_0([0,T], \MC(\C P^n))$ expressed by the r.h.s. of \eqref{eqMFGqu}.

To any curve $\mu_t$ in $C_0([0,T], \MC(\C P^n))$ there corresponds the curve
$\eta_t$ given by \eqref{eqMFGqu1}. It follows that for two curves $\mu_t^1, \mu_t^2$
we have for the corresponding matrix-valued curves $\eta_t^j$ the estimate
\begin{equation}
\label{eq1thonwelposedMFGman}
 {\tr}|\eta_t^1-\eta_t^2|\le  \|\mu^1_.-\mu^2_.\|_{\MC,T}.
\end{equation}

Next we solve the backward HJB equation \eqref{eqHJBIsforspecialhommild}
finding the function $S(t,w)$ that depends on $\mu_.$. Actually $S(t,w)$ depends on
the future $\mu_{\ge t}$ only, but this is not very essential for the argument.

The well-posedness of this HJB equation was proved in \cite{KolDynQuGames}. Namely,
it was shown that for any function $S_T\in C^1(\C P^n)$ (the space of continuously differentiable
functions on $\C P^n)$), there exists a unique curve $S_t \in C([0,T],C^1(\C P^n))$
(the space of continuous curves with values in $C^1(\C P^n)$ that solves \eqref{eqHJBIsforspecialhommild}.

In fact, this well-posedness is a consequence of a general well-posedness result from
\cite{Kolbook19} and the following statement expressing the key smoothing property of the
semigroup $e^{2t \De_{pro}}$:
\begin{equation}
\label{eq1propheatsemomansmoo}
\|e^{2t \De_{pro}}f\|_{C^1(M)}\le C t^{-1/2} \|f\|_{C(M)},
 \end{equation}
 \begin{equation}
\label{eq2propheatsemomansmoo}
\|e^{2t \De_{pro}}f\|_{C^1(M)}\le C \|f\|_{C^1(M)}
 \end{equation}
with a constant $C$, uniformly for any compact interval of time.
Moreover, by the same arguments as in the proof of Theorem 6.1.2 of \cite{Kolbook19}
(devoted to the case of HJB in $\R^d$) it follows that the solution depends Lipschitz
continuously on a parameter, if the corresponding Hamiltonian function depends
Lipschitz continuously on this parameter. In our case the role of the parameter is played
by the matrix-valued curves $\eta_t$, and it follows that for the solutions $S_t^1$ and $S_t^2$
corresponding to the curves $\eta^1_t$ and $\eta_t^2$, the following estimate holds.
\begin{equation}
\label{eq2thonwelposedMFGman}
\max_{t\in [0,T]}\|S_t^1-S_t^2\|_{C^1(\C P^n)}
\le K \max_{t\in [0,T]} {\tr}|\eta_t^1-\eta_t^2|,
\end{equation}
with a constant $K$ depending continuously on the tensor $A$, the time interval $T$ and the norm $\|S_T\|_{C^1(\C P^n)}$.

 Next, by \eqref{eq1thonwelposedMFGman} and \eqref{eq2thonwelposedMFGman}, the square bracket on the r.h.s.
 of \eqref{eqMFGqu} depend Lipschitz continuously on $\mu_.$, and by \eqref{eq1propheatsemomansmoo},
 the derivatives $(\pa/\pa x_k)[\cdots]$ and $(\pa/\pa y_k)[\cdots]$ are of order $(t-s)^{-1/2}$.
 Hence, for two curves $\mu^1_.$ and $\mu^2_.$ we get
 \[
 \|[\Phi(\mu^1_.)](t)-[\Phi(\mu^2_.)](t)\|_{\MC,T}
 \le \sqrt T K_T\|\mu^1_.-\mu^2_.\|_{\MC,T}
 \]
 with $K_T$ depending continuously  on $H$, $A$,$\hat H$, $\ka$  and the norm $\|S_T\|_{C^1(\C P^n)}$.
 For sufficiently small $T$, we get $\sqrt T K_T<1$, and for this $T$ by the Banach contraction principle
 we derive the existence and uniqueness of the fixed point thus proving part (i) of the theorem.

 (ii) Let us now consider $C_0([0,T], \MC(\C P^n))$ with a different topology. Namely,
 the distance in $\MC(\C P^n)$ will be defined from the space $(C^1(\C P^n))^*$, dual
 to the space of smooth functions on $\MC(\C P^n)$. This distance defines the weak topology,
 and therefore $C_0([0,T], \MC(\C P^n))$ becomes a closed convex subset of the set of continuous
 functions from $[0,T]$ to the compact space of probability measures on the compact space $\C P^n$.
 The mapping  $\Phi(\mu_.)$ from part (i) clearly maps $C_0([0,T], \MC(\C P^n))$ into itself.
 To deduce the existence of a fixed point form the Schauder fixed point principle we have
 to show that the image of $\Phi$ is compact in $C_0([0,T], \MC(\C P^n))$. But as seen directly
 from the definition of $\Phi$ this image consists of uniformly $1/2$-H\"older curves, and such curves form
 a compact set due to the Arzela theorem.

\end{proof}

\section{Conclusion}
\label{secconc}

We have developed a new framework for studying the control problems of continuously observed
system of large number of interacting quantum particles, by developing quantum analog
for the theory of mean-field games.
By-passing we proved a rigorous convergence result deriving a limiting nonlinear
stochastic equation for individual particles from a large stochastic system of interacting particles.

Let us point out to some questions and open problems
 arising from this developments.

 The key simplification of our analysis is the assumption of conservativity of observation
 (coupling operators to measuring devices are anti-Hermitian). This an exceptional case, and it
 is of interest to get rid of this restrictive assumption.

 We developed the theory on the assumption of the existence of solutions to the limiting MFG problems.
 The existence was proved only for finite-dimensional state spaces of single particle
 (and only for special homodyne arrangements). The method was based
 on forward-backward system on manifolds. For the standard infinite-dimensional cases such systems
 would become systems of equations in variational derivatives. Therefore possible other methods
 (like stochastic Pontryagin maximum principle) can be used here to establish existence of solutions
 to the limiting MFG problem. Another restriction was the assumption of identical initial states for
  all agents. It would be natural to extend the theory beyond this restriction.

We developed the theory for diffusive type of observation, filtering and control. 
There should be a counterpart of this theory for the observations of counting type.   

Finally, the development of numeric schemes for solving forward-backward MFG systems on manifolds
would be of interest. They can be based, for instance, on some extensions of the technique from \cite{McEa09}.


\begin{thebibliography}{99}


\bibitem{Armen02Adaptive}
M. A. Armen, J. K. Au, J. K. Stockton, A. C. Doherty and H. Mabuchi. Adaptive homodyne measurement
 of optical phase. Phys. Rev. Lett. 89 (2002), 133602.

\bibitem{BarbRock18}
V. Barbu, M. R\"ockner and D. Zhang. Optimal bilinear control of nonlinear stochastic
Schr\"odinger equations driven by linear multiplicative noise.
The Annals of Probability 46:4
(2018), 1957 - 1999.

\bibitem{BarbRock16}
V. Barbu, M R\"ockner and D. Zhang. Stochastic nonlinear Schrödinger equations.
Nonlinear Anal. 136 (2016), 168 - 194.


\bibitem{BarchBel}
A. Barchielli and V.P. Belavkin. Measurements contunuous in time and a posteriori states in quantum mechanics.
J. Phys A: Math. Gen. 24 (1991), 1495-1514.

\bibitem{BarchBook}
A. Barchielli and M. Gregoratti. Quantum Trajectories and Measurements in Continuous Case. The Diffusive Case.
 Lecture Notes Physics, v. 782, Springer Verlag, Berlin, 2009.

\bibitem{Barchonm}
A. Barchielli and A. M.  Paganoni.
On the asymptotic behaviour of some stochastic differential equations for quantum states.
Infin. Dimens. Anal. Quantum Probab. Relat. Top. 6:2 (2003), 223–243.

\bibitem{Bel87}
 V. P. Belavkin, Nondemolition measurement and control in quantum dynamical systems.
 In: Information Complexity and Control in Quantum Physics.
CISM Courses and Lectures 294, S. Diner and G. Lochak, eds., Springer-Verlag, Vienna, 1987, pp. 331–336.

\bibitem{Bel88}
V.P. Belavkin. Nondemolition stochastic calculus in Fock space and nonlinear
filtering and control in quantum systems. Proceedings XXIV Karpacz winter school
(R. Guelerak and W. Karwowski, eds.), Stochastic methods in mathematics and physics.
 World Scientific, Singapore, 1988, pp. 310 - 324.

\bibitem{Bel92} V.P. Belavkin. Quantum stochastic calculus and quantum nonlinear filtering.
 J. Multivar. Anal. 42 (1992), 171 - 201.

 \bibitem{BelKol}
V.P. Belavkin, V.N. Kolokoltsov. Stochastic
evolution as interaction representation of a boundary value
problem for Dirac type equation. Infinite Dimensional Analysis,
Quantum Probability and Related Fields {\bf 5:1} (2002), 61-92.

\bibitem{BenFr}
A. Bensoussan, J. Frehse, P. Yam. Mean Field Games and Mean Field Type Control Theory,
Springer, 2013.


\bibitem{BoutHanQuantumSepar}
L. Bouten and R. Van Handel. On the separation principle of quantum control (2006).
ArXive: math-ph/0511021v2.

\bibitem{BoutHanJamQuantFilt}
L. Bouten, R. Van Handel and M. James. An introduction to quantum filtering.
SIAM J. Control Optim. 46:6 (2007), 2199-2241.

\bibitem{Brzes14}
Z. Brze\'zniak and A. Millet. On the stochastic Strichartz estimates and the stochastic nonlinear Schr\"odinger equation on a
compact Riemannian manifold, Potential Anal. 41 (2) (2014) 269–315.

 \bibitem{Bushev06Adaptive}
 P. Bushev et al. Feedback cooling of a singe trapped ion. Phys. Rev. Lett. 96 (2006), 043003.

\bibitem{CarDelbook18} R. Carmona and F. Delarue. Probabilistic Theory
of Mean Field Games with Applications, v. I, II.
Probability Theory and Stochastic Modelling v. 83, 84. Springer, 2018.


\bibitem{EWL99}
 J. Eisert, M. Wilkens and M. Lewenstein.
Quantum Games and Quantum Strategies. Phys Rev Lett 83:15 (1999), 3077 - 3080.


\bibitem{GolsePaul}
F. Golse and Th. Paul.
Empirical Measures and Quantum Mechanics:
Applications to the Mean-Field Limit.
Commun. Math. Phys. 369, 1021 - 1053 (2019).

\bibitem{Gomesbook} D. Gomes et al.
Regularity theory for mean-field game systems.
Springer 2016.

\bibitem{GreckGenerNonlinSchr}
W. Grecksch and H. Lisei. Stochastic nonlinear equations of Schrödinger type.
Stoch. Anal. Appl. 29:4 (2011), 631 - 653.

\bibitem{GuoZhang08}
 H. Guo, J. Zhang and G. J. Koehler. A survey of quantum games.
 Decision Support Systems 46 (2008), 318-332.

 \bibitem{Holevo91} A.S. Holevo. Statistical Inference for quantum processes.
 In: Quanum Aspects of Optical communications.
Springer LNP 378 (1991), 127-137, Berlin, Springer.

\bibitem{HCM3} M. Huang, R. Malham\'e, P. Caines. Large population stochastic dynamic games:
closed-loop Mckean-Vlasov systems and the Nash certainty equivalence principle.
Communications in information and systems 6 (2006), 221 -- 252.

\bibitem{KhanReview18}
 F. S. Khan, N. Solmeyer, R. Balu and T. Humble. Quantum games:
a review of the history, current state, and interpretation.
Quantum Information Processing (2018),  17:309.

\bibitem{KnowlesPickl}
A. Knowles and P. Pickl. Mean-field dynamics: singular potentials and rate of convergence. Commun. Math.
Phys. 298 (2010), 101–138.



\bibitem{Kol92}
 V.N. Kolokoltsov.
 The stochastic Bellman equation as a nonlinear equation
in Maslov spaces. Perturbation theory.  Dokl. Akad. Nauk
{\bf 323:2} (1992), 223-228.
Engl. transl. in Sov. Math. Dokl. {\bf 45:2} (1992), 294-300.

\bibitem{Kolbook19}
 V. N. Kolokoltsov. Differential equations on measures and functional spaces.
Birkh\"auser Advanced Texts, Birkh\"auser, 2019.

\bibitem{KolQSurv}
V. N. Kolokoltsov. Quantum games: a survey for mathematicians (2019). https://arxiv.org/abs/1909.04466

\bibitem{KolDynQuGames}
Vassili N. Kolokoltsov. Dynamic Quantum Games.
https://arXiv:2002.00271

\bibitem{KolMal} V. N. Kolokoltsov and O.A. Malafeyev. Understanding Game theory.
World scientific, 2010.

\bibitem{KolMalbook19} V. N. Kolokoltsov and O. A. Malafeyev.
Many Agent Games in Socio-economic Systems:
Corruption, Inspection, Coalition Building, Network Growth, Security.
Springer Series in Operations Research and Financial Engineering, Springer Nature, 2019.

\bibitem{KolTrYa14}
V. Kolokoltsov, M. Troeva and W. Yang. On the rate of convergence
for the mean-field approximation of controlled diffusions with large number of players.
Dyn. Games Appl. 4:2 (2014), 208–230.

\bibitem{KoWe13}
V. Kolokoltsov and W. Yang.
 Existence of solutions to path-dependent kinetic equations and related forward - backward systems.
 Open Journal of Optimization {\bf 2:2}, 39-44 (2013), http://www.scirp.org/journal/ojop/

\bibitem{LL2006}  J-M. Lasry, P-L. Lions. Jeux \`a champ moyen, I. Le cas stationnaire.
Comptes Rendus Mathematique
 Acad. Sci. Paris 343:9 (2006), 619 - 625.

\bibitem{MW00}
L. Marinatto and T. Weber. A quantum approach to static games of
complete information.
Physics Letters A 272 (2000), 291-303.

\bibitem{McEa09}
W.M. McEneaney and L.J. Kluberg. Convergence Rate for a Curse-of-Dimensionality-Free Method for a Class of HJB PDEs.
 SIAM J. Control and Opt. 48 (2009).

 \bibitem{MeyerD99}
 D. A. Meyer. Quantum strategies. Phys Rev Lett 82:5 (1999), 1052-1055.

 \bibitem{Pellegrini} C. Pellegrini.
Poisson and Diffusion Approximation of Stochastic Schr\"odinger Equations with Control.
 Ann. Henri Poincaré 10 (2009), no. 5, 995–1025.

 \bibitem{Pickl}
P. Pickl. A simple derivation of mean-field limits for quantum systems. Lett. Math. Phys. 97,
151 - 164 (2011).

 \bibitem{WiMilburnBook}
H. M.  Wiseman and G. J.  Milburn.
 Quantum measurement and control. Cambridge Univesity Press, 2010.


\end{thebibliography}
\end{document}